\documentclass[11pt]{elsarticle}

\usepackage[left=1.00in,top=1.00in,right=1.00in,bottom=1.00in]{geometry}
\usepackage[latin1]{inputenc}
\usepackage{amssymb}
\usepackage{graphicx}
\usepackage{amsmath}
\usepackage{amsthm}
\usepackage{accents}
\usepackage{mathrsfs}
\usepackage{subfigure}

\newtheorem{theorem}{Theorem}
\newtheorem{lemma}[theorem]{Lemma}
\newtheorem{proposition}{Proposition}
\newtheorem{remark}{Remark}

\theoremstyle{definition}
\newtheorem{definition}{Definition}
\newtheorem{assumption}[definition]{Assumption}

\newcommand{\ev}{ \accentset{\star}}
\newcommand{\basisfct}[2]{{\mathrm{B}}^{(\mathit{#1})}_{#2}}
\newcommand{\basisfctSigma}[2]{{\overline{\mathrm{B}}}^{(\mathit{#1})}_{#2}}
\newcommand{\basisfctSigmaTilde}[2]{{\widetilde{\mathrm{B}}}^{(\mathit{#1})}_{#2}}
\newcommand{\basisfctX}[2]{{\ev{\mathrm{B}}}^{(\mathit{#1})}_{#2}}

\newcommand{\basisfctXdual}{\ev{\Lambda}}
\newcommand{\ibasis}{c}
\newcommand{\igf}{\varphi_h}
\newcommand{\Du}{\partial_{1}}
\newcommand{\Dv}{\partial_{2}}

\newcommand{\RR}{\mathbb R}
\newcommand{\f}[1]{\mathbf{#1}}
\newcommand{\fn}{\mathbf{d}}

\newcommand{\V}{\mathcal{V}}

\newcommand{\graph}{\Phi}
\newcommand{\II}{\mathbb{I}}
\newcommand{\s}{\scriptstyle}

\newcommand{\indexOmega}{\mathcal{I}_{\Omega}}
\newcommand{\indexSigma}{\mathcal{I}_{\Sigma}}
\newcommand{\indexSigmaint}{\mathcal{I}^\circ_{\Sigma}}
\newcommand{\indexSigmabound}{\mathcal{I}^{\Gamma}_{\Sigma}}
\newcommand{\indexX}{\mathcal{I}_{\mathcal{X}}}
\newcommand{\indexXint}{\mathcal{I}^\circ_{\mathcal{X}}}
\newcommand{\indexXbound}{\mathcal{I}^{\Gamma}_{\mathcal{X}}}

\newcommand{\tupleSigma}[1]{\mathcal{T}_{\Sigma^{(#1)}}}
\newcommand{\tupleX}[1]{\mathcal{T}_{\f x^{(#1)}}}
\newcommand{\cupdot}{\,\dot{\cup}\,}

\newcommand{\Spr}{\mathcal{S}^{\f p,\f r}_h}
\newcommand{\Spru}{\mathcal{S}^{p,r}_h}

\newcommand{\W}{\mathcal{A}}
\newcommand{\lW}[1]{\W^{(\ii_{#1})}_{{\f x}^{(i)}}}

\newcommand{\ii}{\imath}
\newcommand{\kk}{\kappa}

\newcommand{\A}{\mathcal{A}}

\usepackage{color}

\begin{document}

\begin{frontmatter}
  
\title{The Argyris isogeometric space on unstructured
  multi-patch planar domains}

\author[1]{Mario Kapl}
\author[2,3]{Giancarlo Sangalli}
\author[4]{Thomas Takacs}

\address[1]{Johann Radon Institute for Computational and Applied Mathematics, Austrian Academy of Sciences, Austria}
\address[2]{Dipartimento di Matematica ``F. Casorati'', Universit\`a degli Studi di Pavia, Italy}
\address[3]{Istituto di Matematica Applicata e Tecnologie Informatiche
  ``E. Magenes'' (CNR), Italy}
\address[4]{Institute of Applied Geometry, Johannes Kepler University Linz, Austria}

\begin{abstract} 
Multi-patch spline parametrizations are used in geometric design and isogeometric
  analysis to represent complex domains. 
We deal with a particular class of $C^0$  planar multi-patch spline parametrizations called analysis-suitable $G^1$ (AS-$G^{1}$) multi-patch parametrizations (cf. \cite{CoSaTa16}). 
This class of parametrizations has to satisfy specific geometric
continuity constraints, and is of importance since it allows to
construct, on the multi-patch domain, $C^1$ isogeometric spaces with optimal approximation properties. It was demonstrated in \cite{KaSaTa17b} that AS-$G^1$ multi-patch parametrizations 
are suitable for modeling 
complex planar multi-patch domains. 

In this work, we construct a basis, and an associated dual basis, for
a specific $C^1$ isogeometric spline space~$\W$ over a
given AS-$G^1$ multi-patch parametrization. {We call the space $\W$ the Argyris isogeometric space, since it is $C^1$ across interfaces and $C^2$ at all vertices and generalizes the 
idea of Argyris finite elements (see \cite{argyris1968tuba}) to tensor-product splines.} The considered space~$\W$ is a subspace of the entire $C^1$ isogeometric space~$\mathcal{V}^{1}$, 
which maintains the reproduction properties of traces and normal derivatives along the interfaces. {Moreover, it reproduces all derivatives up to second order at the vertices.}
In contrast to $\mathcal{V}^{1}$, the dimension of 
$\W$  does not depend on the domain parametrization, and
$\W$  admits a  basis and dual
basis which  possess a simple explicit 
representation and   local support. 

{We conclude the paper with some numerical experiments, which exhibit the optimal approximation order of the Argyris isogeometric space $\W$ and demonstrate the applicability of our approach for isogeometric analysis.}
\end{abstract}
\end{frontmatter}

\section{Introduction} \label{sec:introduction}

Multi-patch spline parametrizations are a powerful tool in computer-aided geometric design for modeling complex domains (cf. \cite{Fa97,HoLa93}). In the framework of isogeometric 
analysis (IGA) (cf. \cite{ANU:9260759,CottrellBook,HuCoBa04}) the underlying spline spaces of these parametrizations are used to define (smooth) discretization spaces for numerically 
solving partial differential equations (PDEs) over the multi-patch domains. When solving a fourth order PDE, such as the biharmonic equation, e.g. 
\cite{BaDe15, CoSaTa16, KaBuBeJu16, KaViJu15, TaDe14},  the Kirchhoff-Love plate/shell problem, e.g. 
\cite{ABBLRS-stream,benson2011large,kiendl-bazilevs-hsu-wuechner-bletzinger-10,kiendl-bletzinger-linhard-09,KiHsWuRe15}, or the Cahn-Hilliard equation, e.g. 
\cite{gomez2008isogeometric,GoCaHu09,LiDeEvBoHu13}, by means of its
weak formulation using a standard Galerkin projection, approximating functions of
global $C^1$-smoothness are needed.

In this work we construct an isogeometric space with global  $C^1$
regularity over an unstructured   multi-patch
parametrization. Our construction takes inspiration from   the
Argyris finite element \cite{argyris1968tuba}, which has been
the progenitor of all the $C^1$ triangular finite elements. The
original Argyris construction uses polynomials of total degree $5$ and
the following degrees-of-freedom: The function values, first and
second derivatives at the element vertices, and the normal derivatives at the
element edge midpoints. These degrees-of-freedom  determine
the trace at the  edges (a polynomial of degree $5$ on each edge) and the normal
derivatives at the edges (a polynomial of degree $4$), and by that, 
determine the polynomial on the whole triangle (see Figure
\ref{fig:Argyris-triangle-and-patch-a}).  $C^1$
regularity follows from the gluing of the normal derivative at the edges. 

After Argyris, other $C^1$ constructions on triangular meshes have been
proposed,  in the context of finite elements
(see  for example the book \cite{ciarlet2002finite}).  The Argyris triangular
element has been used only recently for
surface parametrizations, in \cite{jaklivc2017hermite}. However, 
splines on triangular meshes are commonly used in geometric design,  often with  $C^2$ smoothness at the mesh
vertices, see the book \cite{lai2007spline} for references.

In the finite element literature, the results on
 $C^1$  quadrilateral elements are restricted  to meshes of structured
rectangles, where the first and most well-known construction is the
Bogner-Fox-Schmit element, introduced in \cite{bogner1965generation}.
Isogeometric analysis has been  reinvigorating this
interest: The recent papers
\cite{CoSaTa16,KaBuBeJu16,KaSaTa17,KaViJu15} and the book \cite{BeMa14}
shed light on the conditions of global $C^1$ regularity for
isogeometric spaces on unstructured quadrilateral meshes,
showing at the same time that these spaces are difficult to characterize. This
motivates our present work: We  design a basis and the related degrees-of-freedom
for a subspace of the complete  $C^1$ isogeometric
space, mimicking  the Argyris construction. The space we propose is 
therefore  named \emph{Argyris isogeometric space}  and denoted by $\A$.

Consider a quadrilateral element  $\bar \Omega^{(i)} =\f F^{(i)} ([0,1]^2)$, i.e., given
by a bilinear mapping of the reference square element. The  lowest-degree  polynomial version
of  $\A$  contains functions $\varphi_h$ such that $\varphi_h \circ \f F^{(i)}$
is a (bi)quintic polynomial,
and $ (\nabla  \varphi_h\cdot \f d) \circ \f F^{(i)}$ is a quartic polynomial
on the edges of the reference element. There are two main differences with
respect to the original Argyris triangle. The first is obvious:  $\varphi_h |_{\bar
  \Omega^{(i)}} $  is not polynomial since $\f F^{(i)} $ is not
linear,  in general.  The second is technical:  $  \f d$ is not the
normal unitary vector to the
edges of $\bar \Omega^{(i)} $, it is instead a suitable  non-constant
direction,  shared with the adjacent
quadrilateral and dependent on it. The degrees-of-freedom are indeed:
\begin{itemize}
\item $24$ (vertex) degrees-of-freedom giving the value, first and second derivatives at each vertex, fully
  determining the trace of $\varphi_h$ at the boundary of the
  element,
\item  $4$ (edge) degrees-of-freedom, one per edge, that with the previous
  information fully determine $(\nabla  \varphi_h\cdot \f d) \circ \f F^{(i)}$  as a quartic polynomial
on each edge,
\item $4$ (interior) degrees-of-freedom that, with the trace and
  directional derivative given at the boundary,  fully
  determine $\varphi_h$.
\end{itemize}
The total number of degrees-of-freedom is $32$ and they are depicted in  Figure
\ref{fig:Argyris-triangle-and-patch-c}. However the condition
that $(\nabla  \varphi_h\cdot \f d) \circ \f F ^{(i)}$ is a quartic polynomial on
each edge is responsible for $4$ additional scalar constraints due to degree elevation to quintic polynomials. 
The degrees-of-freedom together with the constraints give $36$
conditions, which matches the dimension of the (bi)quintic polynomial space.
We remark that the boundary degrees-of-freedom above correspond to the ones
of the original Argyris triangular element, with  $  \f d$ replacing  $  \f n$
(as for the Argyris triangular element, these edge degrees-of-freedom depend on
the mesh and do not admit a universal definition on the reference
element).    Interior degrees-of-freedom appear also
in the higher-degree  Argyris triangular element, see Figure
\ref{fig:Argyris-triangle-and-patch-b}. Our
construction is generalizable to any degree $p\geq 5$, see  Figure
\ref{fig:Argyris-triangle-and-patch-d}.

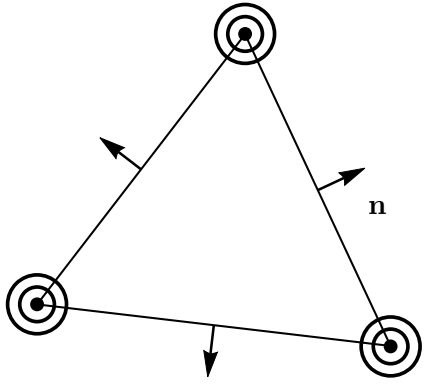
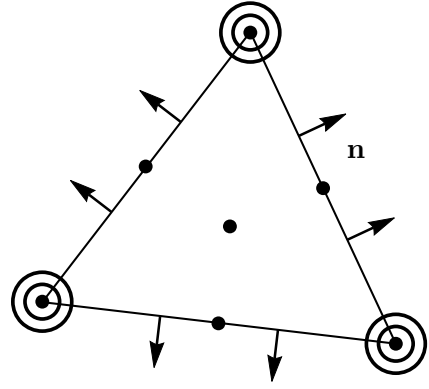
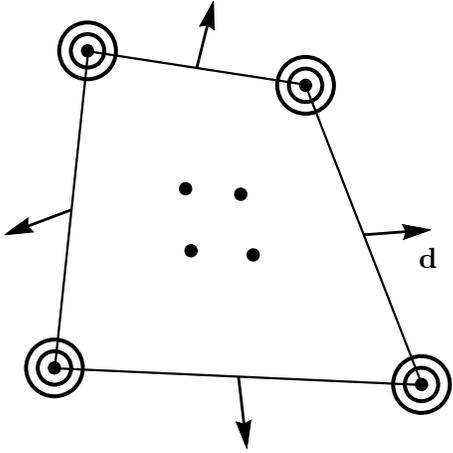
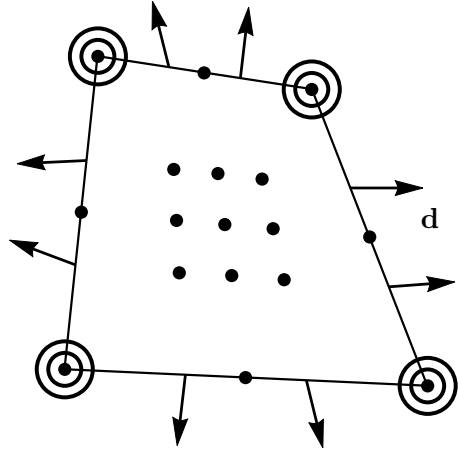
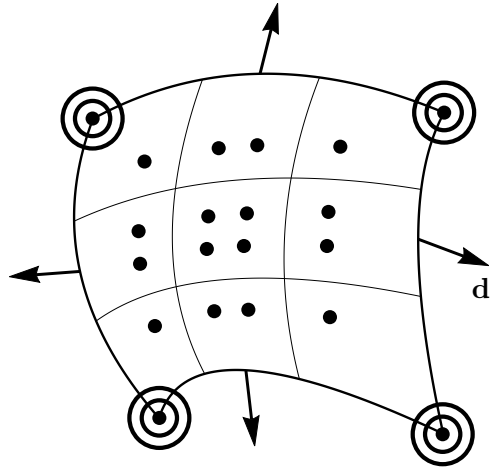
\begin{figure}
 \centering
 \subfigure[Triangular element for degree $5$.]{\begin{picture}(200,100)
    \put(0,0){\includegraphics[width=0.37\textwidth]{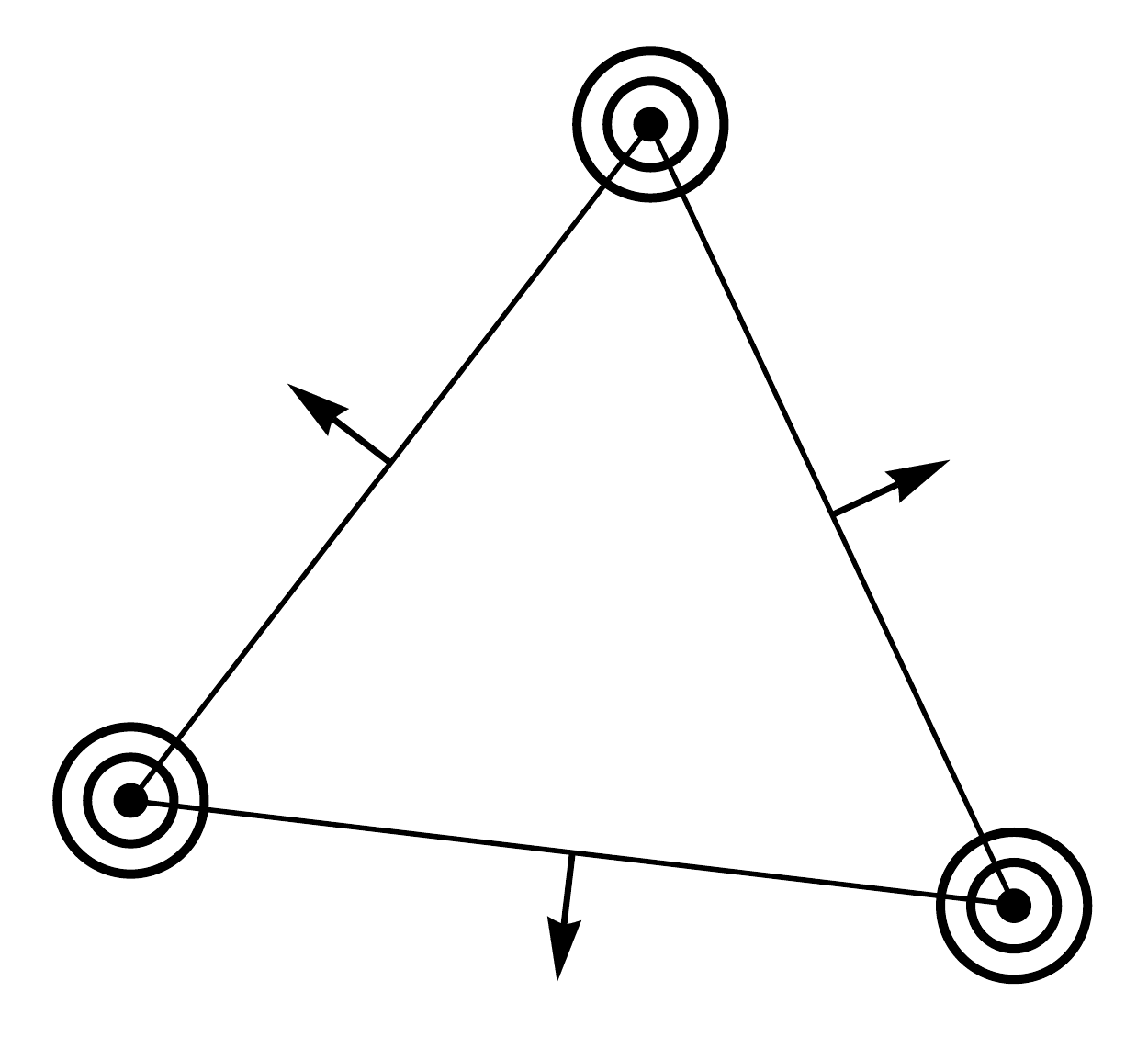}\label{fig:Argyris-triangle-and-patch-a}}
    \put(145,70){$\f n$}
  \end{picture}}
 \hspace{0.1\textwidth}
 \subfigure[Triangular element for degree $6$.]{\begin{picture}(190,110)
    \put(0,0){\includegraphics[width=0.37\textwidth]{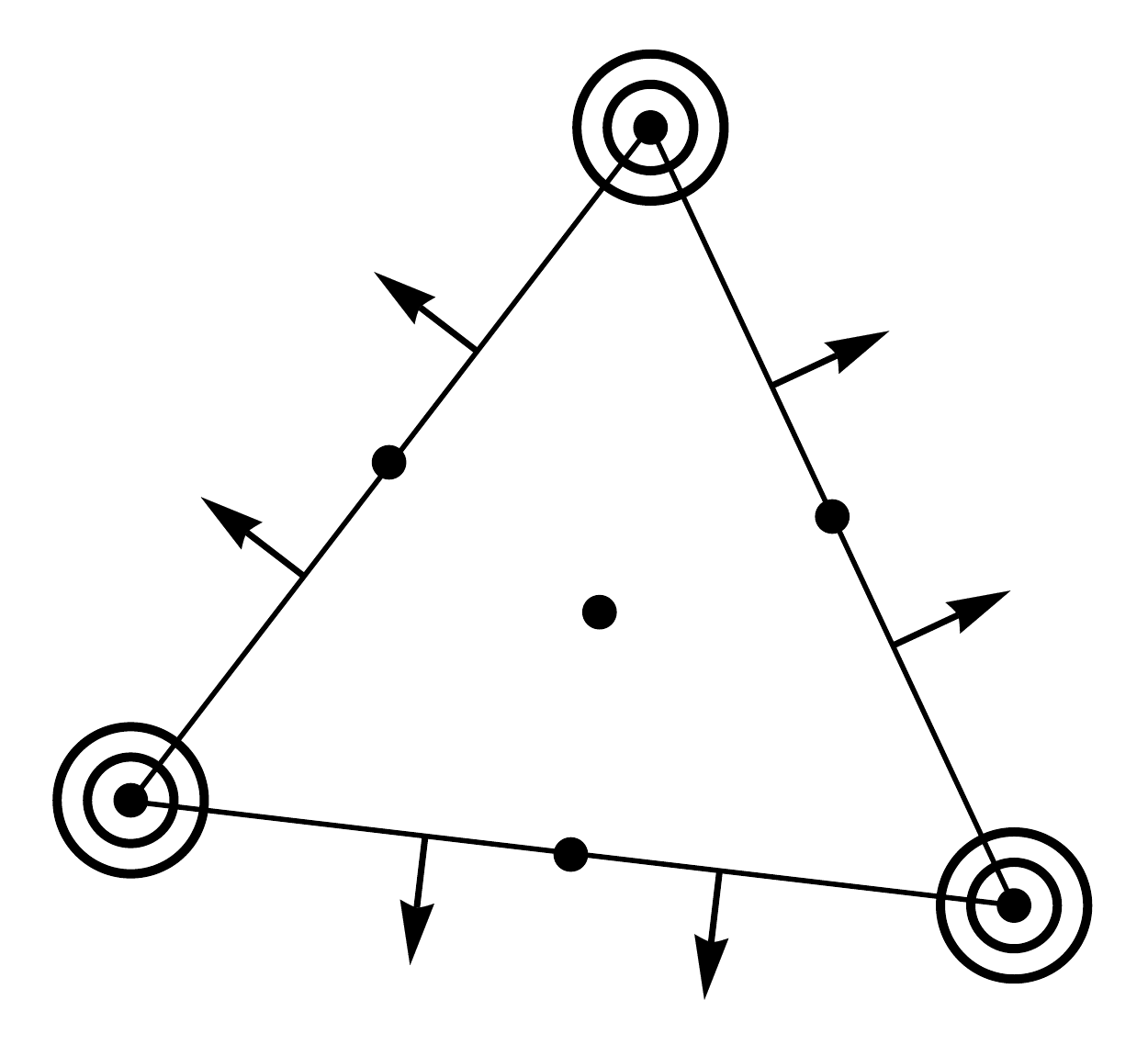}\label{fig:Argyris-triangle-and-patch-b}}
    \put(135,93){$\f n$}
  \end{picture}}
 \\
 \vspace{20pt}
 \subfigure[Quadrilateral element for bidegree $5$.]{\begin{picture}(190,180)
    \put(0,0){\includegraphics[width=0.4\textwidth]{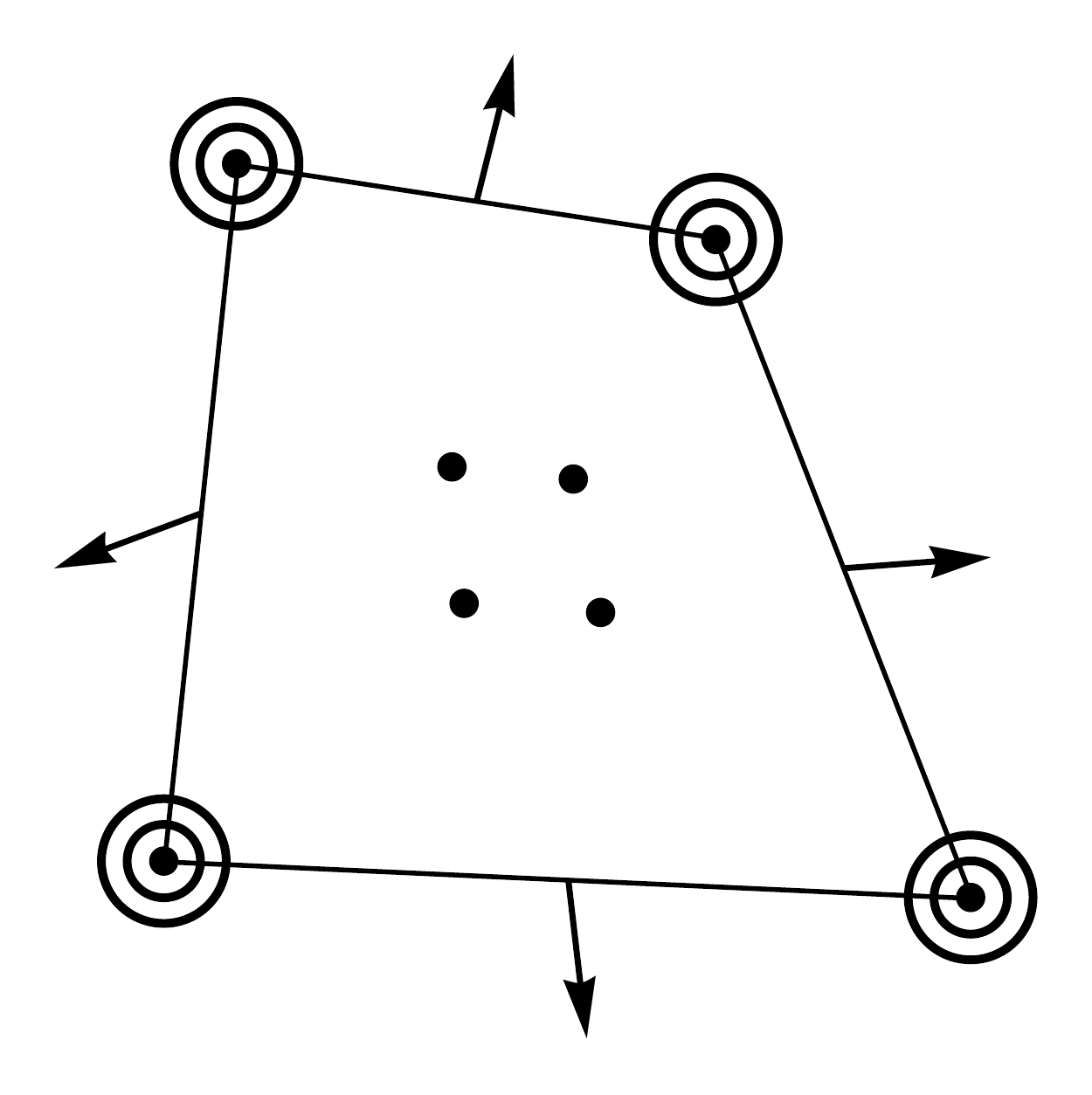}\label{fig:Argyris-triangle-and-patch-c}}
    \put(166,78){$\f d$}
  \end{picture}}
 \hspace{0.1\textwidth}
 \subfigure[Quadrilateral element for bidegree $6$.]{\begin{picture}(190,180)
    \put(0,0){\includegraphics[width=0.4\textwidth]{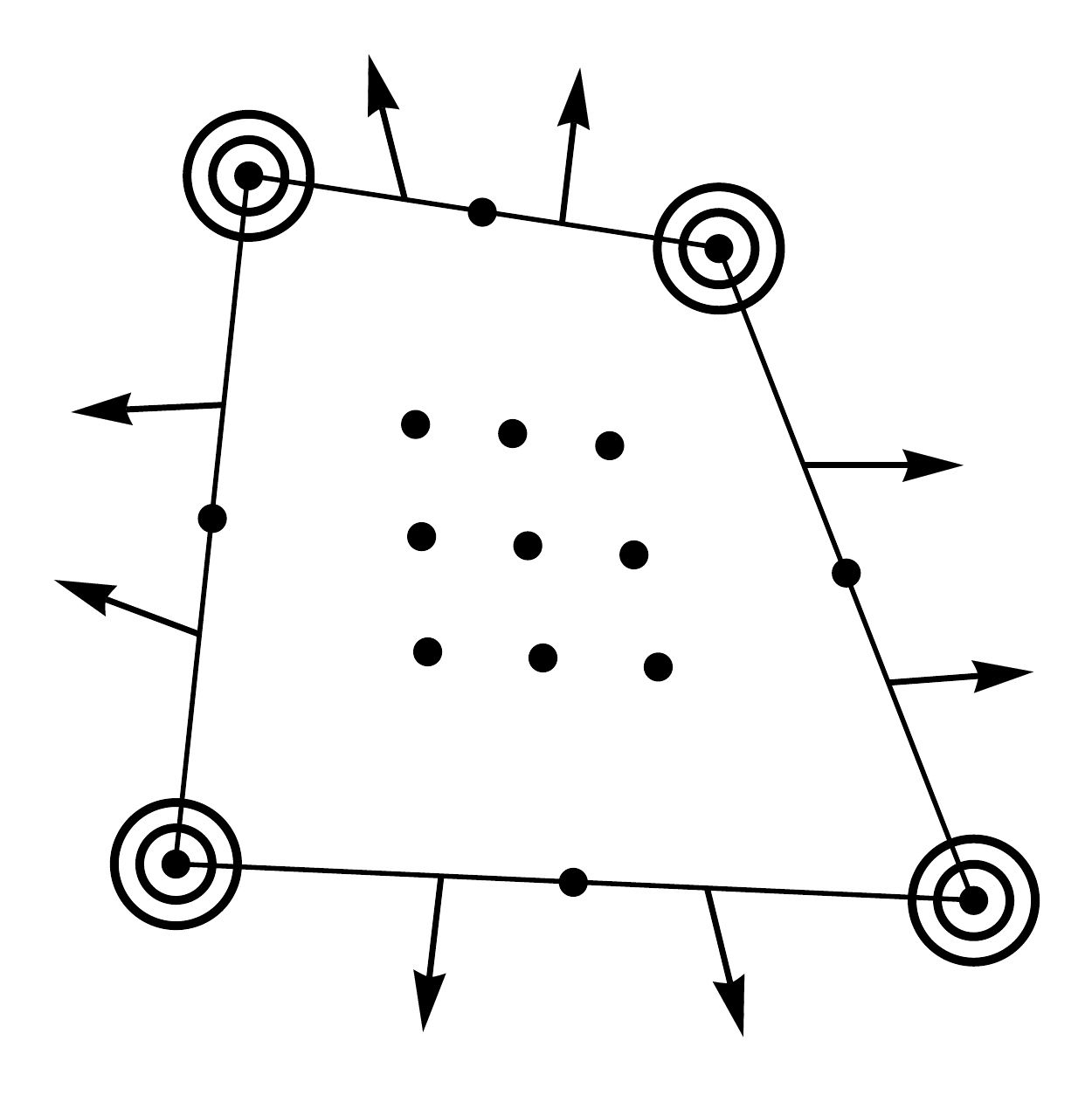}\label{fig:Argyris-triangle-and-patch-d}}
    \put(165,93){$\f d$}
  \end{picture}}\\
 \vspace{20pt}
 \subfigure[Isogeometric element for bidegree $3$.]{\begin{picture}(190,180)
    \put(0,0){\includegraphics[width=0.43\textwidth]{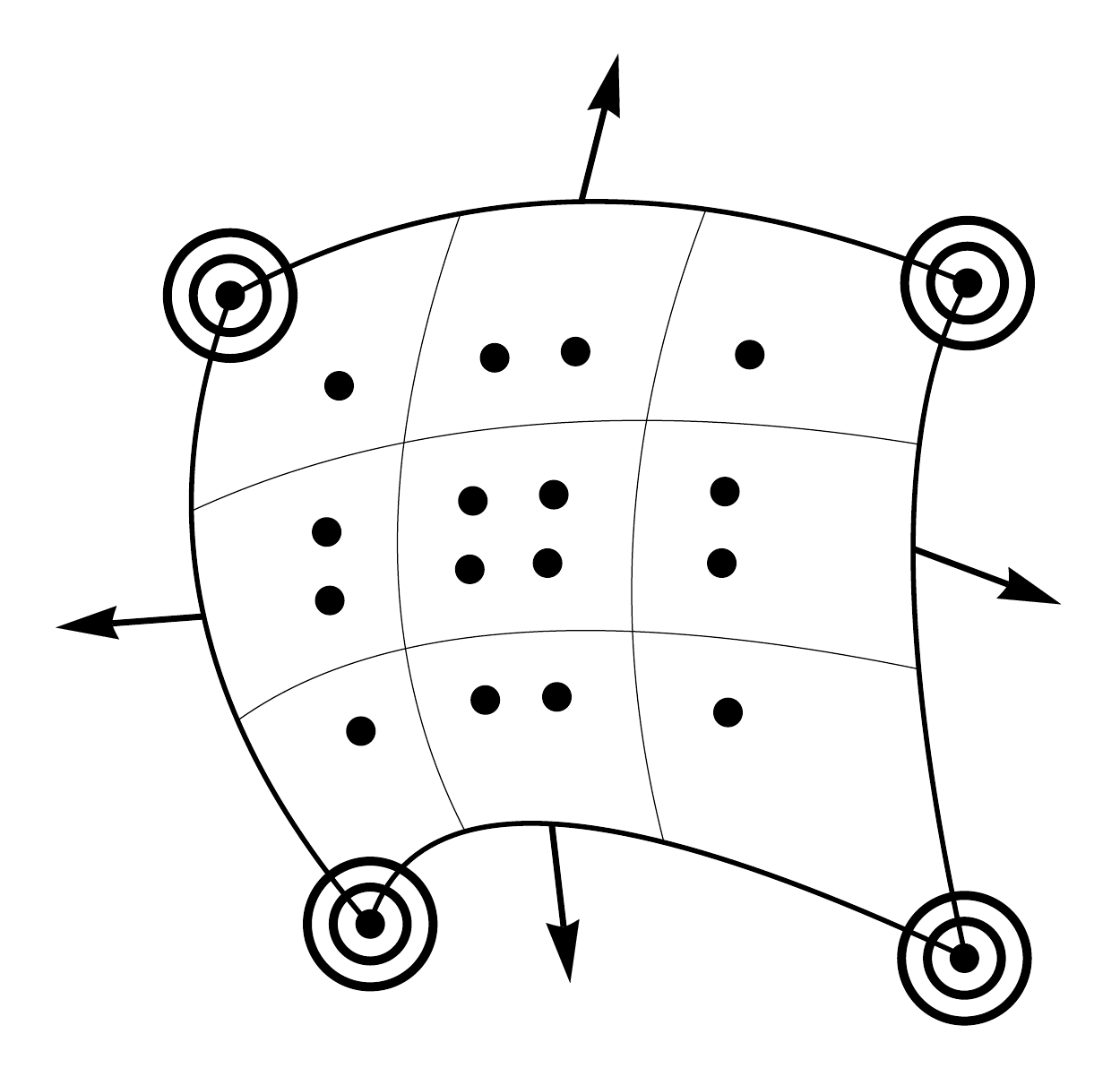}\label{fig:Argyris-triangle-and-patch-e}}
    \put(185,73){$\f d$}
  \end{picture}}
  \caption{Argyris-type finite elements, with associated vertex,
    edge and interior degrees-of-freedom. Note that for the
    quadrilateral and isogeometric elements the derivative in normal
    direction $\f n$ has to be replaced by the derivative in a
    transversal direction $\f d$ (see Definition \ref{def:edge-space-dual-proj}).}
  \label{fig:Argyris-triangle-and-patch}
\end{figure}

The construction above further extends in two ways. The
first is that, instead of a quadrilateral element, we can allow a
B\'ezier patch, that is $\f F^{(i)} $ can be a $p$-degree polynomial. However there are compatibility conditions
between adjacent patch  parametrizations
that guarantee that the  Argyris isogeometric space does not get
overconstrained: We
need the $\f F^{(i)} $ to form an \emph{analysis-suitable $G^1$} multi-patch
parametrization (see \cite{CoSaTa16}). 
The second extension is from a polynomial space (on each patch) to a
tensor-product spline space. Not only is this important in
isogeometric analysis, it also allows a degree reduction. Indeed we can
construct an  Argyris isogeometric patch from (bi)cubic $C^1$
continuous splines, see  Figure
\ref{fig:Argyris-triangle-and-patch-e}.

In this paper we show that  the Argyris isogeometric space is well
defined by  constructing a suitable basis and
dual basis possessing desirable properties, such as local
support and an explicit representation, which can be 
evaluated and manipulated easily. 

A key ingredient of our approach is the determination of the compatibility
conditions that the parametrizations $\f F^{(i)} $  of the patches
need to fulfill in order to guarantee that the Argyris isogeometric
space possesses  optimal approximation order. This is based on the
mentioned work \cite{CoSaTa16}, where the class of  AS-$G^1$ multi-patch 
parametrizations has been introduced.  The paper \cite{KaSaTa17b} has
then shown  numerically that this  class of  parametrizations enables the geometric design of 
complex planar multi-patch domains. The construction can be
extended to surfaces, and the $C^1$ isogeometric spaces
constructed on  AS-$G^1$ multi-patch parametrizations  exhibit optimal
convergence  under $h$-refinement.

Piecewise bilinear multi-patch parametrizations are a subclass of the
class of AS-$G^1$ multi-patch parametrizations and were  
considered to generate a $C^1$  
basis in~\cite{BeMa14,KaBuBeJu16,KaViJu15}. The focus therein is however
to characterize the full $C^1$ isogeometric space, which we denote
$\mathcal{V}^1$ in this paper. The papers \cite{KaBuBeJu16,KaViJu15} study $\mathcal{V}^1$ for
uniform spline functions of degree~$3$.  The work
\cite{BeMa14} focuses on  B\'{e}zier polynomials of degree~$4$ and
$5$ and generates  basis functions by means of minimal determining sets
(cf.~\cite{LaSch07}) for the involved B\'{e}zier coefficients. These
approaches can be extended to mapped  piecewise bilinear multi-patch parametrizations, which are also AS-$G^1$ and allow to model certain domains with curved boundaries and interfaces, see \cite{KaBuBeJu16,KaViJu15}. 
Still, more general AS-$G^1$ multi-patch parametrizations, such as domains with smooth boundaries, cannot be handled.
In \cite{KaSaTa17} an explicit  basis construction was given allowing  non-uniform isogeometric
spline functions of arbitrary degree $p \geq 3$ and 
regularity (with regularity $r$  up to  $p-2$), but on a two-patch geometry, with AS-$G^1$  parametrization.

Unlike the Argyris isogeometric space
$\A$, the dimension of the full $C^1$ space, that is  $\mathcal{V}^1$, depends on the domain
parametrization, see \cite{KaSaTa17}.
 In fact $\A$  possesses a simpler structure than  $\mathcal{V}^1$, but maintains its reproduction properties
  for traces and normal derivatives along the
  interfaces. In the present work we only provide  numerical evidence of
  the optimal approximation properties of $\A$, postponing the mathematical analysis to a
  further work.

In our setting, and in the papers
\cite{BeMa14,BlMoVi17,CoSaTa16,KaBuBeJu16,KaSaTa17,KaViJu15,mourrain2015geometrically},
the $C^1$ isogeometric functions are defined over a
domain given by a multi-patch parametrization which is not $C^1$ at the
patch interfaces. However there is another possibility, that is, when  the multi-patch
parametrization is  $C^1$  everywhere except in the vicinity 
of an extraordinary vertex, where the parametrization is singular.
$C^1$  isogeometric spaces in this case are constructed and
studied in  \cite{KaNgPe17,Peters2,NgKaPe15,ScSiEv13,ToSpHu17b,ToSpHu17,WuMoGaNk17}. A special case is
when the parametrization is polar at the extraordinary vertex, see
\cite{Ta2014,TaJu2012,ToSpHiHu16}. 

The remainder of the paper is organized as follows. Section~\ref{sec:preliminaries} presents some basic definitions and notations which are used throughout the paper. This includes 
the presentation of the spline spaces and of the multi-patch domain parametrizations as well as the local and global indexing for the patches, edges and vertices that we use. 
In Section~\ref{sec:c1-iga-spaces} we recall the concept of AS-$G^1$
multi-patch parametrizations and  the framework of $C^1$  isogeometric spline spaces over 
this class of multi-patch parametrizations. Section~\ref{sec:W}
describes the construction of a basis and of its associated dual basis
for the  Argyris isogeometric space $\A$. In Section \ref{sec:tests}
we perform  $L^2$-approximation over different AS-$G^1$
parametrizations to demonstrate the potential of the  Argyris
isogeometric space $\A$ 
for applications in IGA.  After the concluding
remarks in  Section~\ref{sec:conclusion}, we deliver technical proofs in \ref{appendix:proofs}, and describe  in \ref{appendix:a}
and \ref{sec:another_space} some extensions of our construction. 

\section{Preliminaries}
\label{sec:preliminaries}

We describe the general notation as well as the multi-patch framework, which will be considered and used throughout the paper.
First, in Sections \ref{sec:notation} and \ref{sec:parametrization} we introduce the general notation, the uni- and bivariate B-spline spaces and bases as well as the 
multi-patch domain we consider. Then, 
we recall in Section \ref{sec:global-local-indices} the standard global-to-local index mapping of mesh objects within our framework. 
Finally, we introduce in Section \ref{sec:local-parametrization} a specific local reparametrization, which will simplify the definitions of the smooth basis functions in Section \ref{sec:W}.

\subsection{Basic notation and spline spaces}
\label{sec:notation}

We consider an open domain $\Omega\subset \mathbb{R}^2$, connected and
regular, and   $\Gamma = \partial\Omega$ being its boundary.
 If $\omega\subset \Omega $ is a
manifold of dimension $0$ (a point) or $1$ (a line), we denote by 
$C^k(\omega) $ the set of piecewise smooth functions defined on
$\Omega$ for which the  $k$-order derivatives are continuous at each
point of $\omega$.

We denote by  $\Spru$ the spline space of
degree $p$ and continuity $C^r$ on the parameter domain $[0,1]$, 
constructed from  an open knot vector with $n$ non-empty knot-spans (i.e.,
elements), then having  mesh size $h=1/n$. We restrict here to uniform  knot spans
for simplicity,  see \cite{KaSaTa17} and \ref{appendix:a} for the generalization. 
The 
multiplicity of the interior knots is $p-r$. 
\begin{definition}[Univariate B-spline basis]
Given a integers  $p\geq 1 $, $r \leq p$, and $n
\geq 1$,  
we denote by $\{ b_j\}_{j \in  \{0,1,\ldots,N-1\}}$, with
$N=(p-r)(n-1)+p+1$,  the standard B-spline basis for the $C^r$
univariate $p$-degree polynomial space 
$\mathcal{S}^{p,r}_h$ on $[0,1]$ with  uniform mesh of mesh size
$h=1/n$.
\end{definition}

The tensor-product spline space 
on the parameter (reference) domain $[0,1]^2$ is   $\Spr = \Spru\otimes\Spru$, where $\f p=(p,p)$ and $\f r=(r,r)$ indicate 
double indices which we assume, for the sake of simplicity, to be
the same in the the two directions. 
\begin{remark}
For any $\f  r$, $ \mathcal{S}^{\f p, \f r}_{1}$ is the space of polynomials of degree $\f p$.
\end{remark}
\begin{definition}[Tensor-product B-spline basis]
We denote the standard tensor-product B-spline basis of the space
$\Spr$ by $\{b_{\f j}\}_{\f j \in  \mathbb{I}}$, with $b_{\f
  j}(\xi_1,\xi_2) = b_{j_1}(\xi_1)b_{j_2}(\xi_2)$, where $\f j =
(j_1,j_2)$ and 
$\mathbb{I} = \{0,\ldots,N-1\}\times\{0,\ldots,N-1\} $.
\end{definition}

\begin{assumption}[Minimum regularity within the patches]\label{ass:C1-in-patches}
  We assume $r\geq1$, that is, $\Spr  \subset
  C^1( [0,1]^2 )$. 
\end{assumption}

\subsection{Multi-patch domain parametrization}
\label{sec:parametrization}

We consider a multi-patch domain parametrization, composed of four-sided subdomains, called patches, 
interfaces between those patches as well as vertices, where several interfaces meet.
The index sets containing all patches, all edges and all vertices will be denoted 
by $\indexOmega$, $\indexSigma$ and $\indexX$, respectively. Moreover, we will have
$\indexSigma = \indexSigmaint \cupdot \indexSigmabound$, 
where $\indexSigmaint$ collects all indices representing the patch interfaces and the indices in 
$\indexSigmabound$ represent all boundary edges. Similarly, we will have $\indexX = \indexXint \cupdot \indexXbound$, 
where the indices in $\indexXint$ represent all interior vertices and 
the ones in $\indexXbound$ represent all boundary vertices. To avoid confusion, we will denote all index sets of patches, interfaces and vertices with a calligraphic~$\mathcal{I}$, and all index sets of basis functions with a double 
struck~$\mathbb{I}$.

\begin{assumption}[Multi-patch domain $\Omega$]\label{ass:Omega}
The domain $\Omega$ is the image of a regular  multi-patch spline
parametrization, i.e., 
\begin{equation}\label{eq:multi-patch-Omega} 
	\overline{\Omega} = \bigcup_{i\in\indexOmega} \overline{\Omega^{(i)}},
\end{equation}
where $ \left \{ \Omega^{(i)} \right \}_{i\in\indexOmega} $ is a 
regular and  disjoint partition,  without hanging nodes, and  each $\Omega^{(i)}$ is an open  spline patch,
\begin{equation}
	\f F^{(i)}: [0,1]^2 \rightarrow
        \overline{\Omega^{(i)}}\subset \RR^2 , \label{eq:F}
\end{equation}
where $\f F^{(i)}\in \Spr \times \Spr  $  are  non-singular and orientation-preserving, i.e., for all $i\in\indexOmega$ and
for all $ (\xi_1,\xi_2) \in [0,1]^2$, it holds
\begin{equation}\label{eq:non-singular-F}
\det  \left [\begin{array}{ll}
           \Du \f F^{(i)} (\xi_1,\xi_2) &   \Dv  \f F^{(i)}(\xi_1,\xi_2) \\
         \end{array}\right ]> 0.
\end{equation}
\end{assumption}

The domain $\Omega$ is partitioned into the union of patches, interfaces and interior vertices
\begin{displaymath}
   {\Omega} = \left(\bigcup_{i\in\indexOmega} \Omega^{(i)} \right)
    \cup \left(\bigcup_{i\in\indexSigmaint} \Sigma^{(i)} \right)
    \cup \left(\bigcup_{i\in\indexXint} \f x^{(i)}\right).
\end{displaymath}
The boundary $\Gamma$ of the domain is given by the collection of boundary edges 
and boundary vertices
\begin{displaymath}
   \Gamma = \left(\bigcup_{i\in\indexSigmabound} \Sigma^{(i)} \right)
    \cup \left(\bigcup_{i\in\indexXbound} \f x^{(i)}\right).
\end{displaymath}

\subsection{Global to local index conversion for edges and vertices}
\label{sec:global-local-indices}

Each edge and vertex of the multi-patch
partition corresponds to a \emph{global index} $i\in\indexSigma$ or $i\in\indexX$, respectively. 
In order to relate the edges and vertices to the patch parametrizations, each global index will 
be associated with a set of \emph{local indices}, depending on the patches that share the edge or vertex.

\begin{figure}[ht]
  \centering
  \begin{picture}(150,140)
    \put(0,0){\includegraphics[width=.3\textwidth]{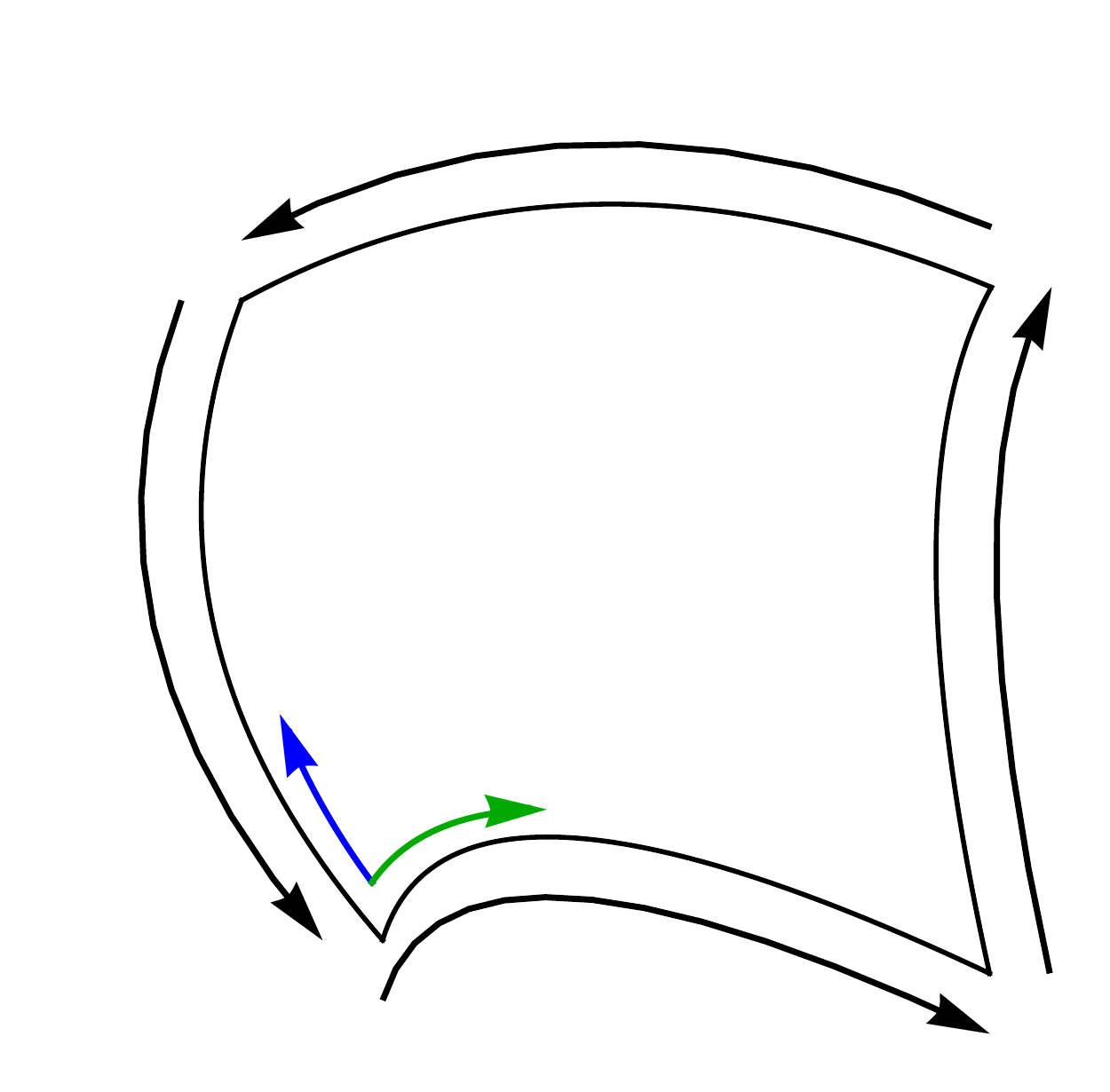}}
    \put(70,63){${\Omega}^{(\ii)}$}
    \put(20,4){${\f x}^{(\ii,0)}$}
    \put(130,4){${\f x}^{(\ii,1)}$}
    \put(130,110){${\f x}^{(\ii,2)}$}
    \put(10,110){${\f x}^{(\ii,3)}$}
    \put(-10,63){${\Sigma}^{(\ii,0)}$}
    \put(70,0){${\Sigma}^{(\ii,1)}$}
    \put(140,63){${\Sigma}^{(\ii,2)}$}
    \put(70,130){${\Sigma}^{(\ii,3)}$}
  \end{picture}
  \caption{Local indexing of edges and vertices for the patch
    ${\Omega}^{(\ii)}$. The parametrizations
    \eqref{eq:local-indexing-edges} induce a counterclockwise
    orientation of the edges. Here, the green and blue arrows represent 
the local coordinate system for the parameters $\xi_1$ and $\xi_2$, respectively.}
  \label{fig:edge-vertex-numbering}
\end{figure}
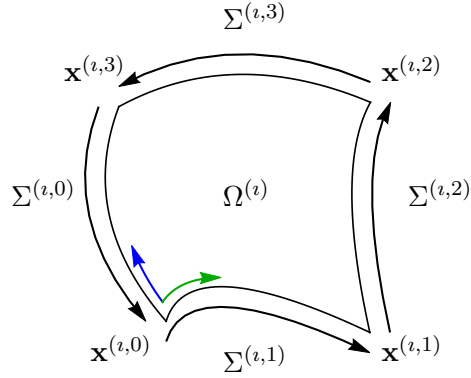
The local index, which we define in the following, 
is in fact a multi-index $(\ii,\kk)$, comprised of a patch index $\ii\in\indexOmega$ and a local numbering $\kk\in\{0,\ldots,3\}$.
\begin{remark}
We always denote with $\ii$ or $\ii_k$ dependent (secondary) indices, which depend on a (primary) index $i$, e.g. $\Omega^{(\ii_1)}$, $\Omega^{(\ii_2)}$ being the patches sharing an interface $\Sigma^{(i)}$.
\end{remark}
The four edges of each  patch $\Omega^{(\ii)}$ are indexed as follows:
\begin{equation}\label{eq:local-indexing-edges}
\begin{array}{ll}
  \Sigma^{(\ii,0)}=\{\f F^{(\ii)}(0,1-\xi):\xi\in]0,1[\},\quad
  &\Sigma^{(\ii,1)}=\{\f F^{(\ii)}(\xi,0): \xi\in]0,1[\}, \\
  \Sigma^{(\ii,2)}=\{\f F^{(\ii)}(1,\xi): \xi\in]0,1[\},
  &\Sigma^{(\ii,3)}=\{\f F^{(\ii)}(1-\xi,1): \xi\in]0,1[\};
\end{array}
\end{equation}
for its four vertices we set: 
\begin{equation*}
\begin{array}{ll}
  {\f x}^{(\ii,0)}=\{\f F^{(\ii)}(0,0)\},\quad
  &{\f x}^{(\ii,1)}=\{\f F^{(\ii)}(1,0)\}, \\
  {\f x}^{(\ii,2)}=\{\f F^{(\ii)}(1,1)\},
  &{\f x}^{(\ii,3)}=\{\f F^{(\ii)}(0,1)\};
\end{array}
\end{equation*}
see Figure \ref{fig:edge-vertex-numbering}. {Here and in what follows, the local coordinate system is depicted with green and blue arrows, corresponding to the $\xi_1$- and $\xi_2$-parameter directions, respectively.}

The  global to local
index conversion for the edges is defined  as follows. 
\begin{assumption}[Global to local index conversion for edges]\label{assu:global-local-edge}
For each global index  $i \in \indexSigmaint$ there exists a
   set  $\tupleSigma{i}=\{(\ii_1,\kk_1),(\ii_2,\kk_2)\}$, with
   $\ii_1,\ii_2\in\indexOmega$, $\ii_1 \neq \ii_2$,  and $\kk_1,\kk_2\in\{0,1,2,3\}$, and 
\begin{equation}
\Sigma^{(i)} = \Sigma^{(\ii_1,\kk_1)} = \Sigma^{(\ii_2,\kk_2)} \subset \Omega.
\label{eq:Sigma-interface}
\end{equation} 
 For each global index   $i \in \indexSigmabound$ we have $\tupleSigma{i}=\{(\ii_1,\kk_1)\}$, with $\ii_1\in\indexOmega$ and $\kk_1\in\{0,1,2,3\}$, and 
\begin{equation}
\Sigma^{(i)} = \Sigma^{(\ii_1,\kk_1)} \subset \Gamma.
\end{equation}  
\end{assumption}
Similarly, we can define the global to local index conversion for 
vertices. 
\begin{assumption}[Global to local index conversion for vertices]\label{assu:global-local-vertex}
For each $i \in \indexX$ there exists a set  $\tupleX{i}=\{(\ii_2,\kk_2),\ldots,(\ii_{2\nu},\kk_{2\nu})\}$, with 
$\ii_2,\ldots,\ii_{2\nu}\in\indexOmega$, being $\nu$ different patch indices, 
and $\kk_2,\ldots,\kk_{2\nu}\in\{0,1,2,3\}$, where 
\begin{equation}
	{\f x}^{(i)} = {\f x}^{(\ii_2,\kk_2)} = \ldots = {\f x}^{(\ii_{2\nu},\kk_{2\nu})}.
\end{equation} 
Here, $\nu$ is the \emph{patch valence} of the vertex  ${\f x}^{(i)}$. 
\end{assumption}
Note that we only consider even sub-indices for all patches. This is because later we introduce odd sub-indices for the adjacent interfaces (see Figure~\ref{fig:reparam} and Definition \ref{def:vertex-standard-form}).
The patch valence 
 coincides with the \emph{edge valence} for interior vertices (this is
 the classical notion of valence).   For boundary vertices the  edge valence is larger by one 
than the patch valence. 

Note that the inverse mapping from local to global indices is unique in the following sense: 
Each local edge index  is associated to a unique global index of  an interface or a boundary edge, i.e., 
for each $\ii_1\in\indexOmega$ and for each $\kk_1\in\{0,1,2,3\}$ there exists exactly one $i\in\indexSigma$, such that $(\ii_1,\kk_1) \in \tupleSigma{i}$. 
Moreover, for each $\ii_2\in\indexOmega$ and for each $\kk_2\in\{0,1,2,3\}$ there exists exactly one $i\in\indexX$, such that $(\ii_2,\kk_2) \in \tupleX{i}$.

\subsection{Parametrization in standard form for edges and vertices}
\label{sec:local-parametrization}

In order to simplify the construction of smooth basis functions across interfaces and vertices, 
we assume that the patch parametrizations are given in \emph{standard form}, as depicted in Figure \ref{fig:reparam}. 
This is obviously not always the case, but can be achieved by reparametrization, as we will 
demonstrate in Lemma \ref{lem:reparam-standard-form}.

\begin{figure}[ht]
  \centering
  \begin{picture}(250,125)
    \put(0,0){\includegraphics[width=.45\textwidth]{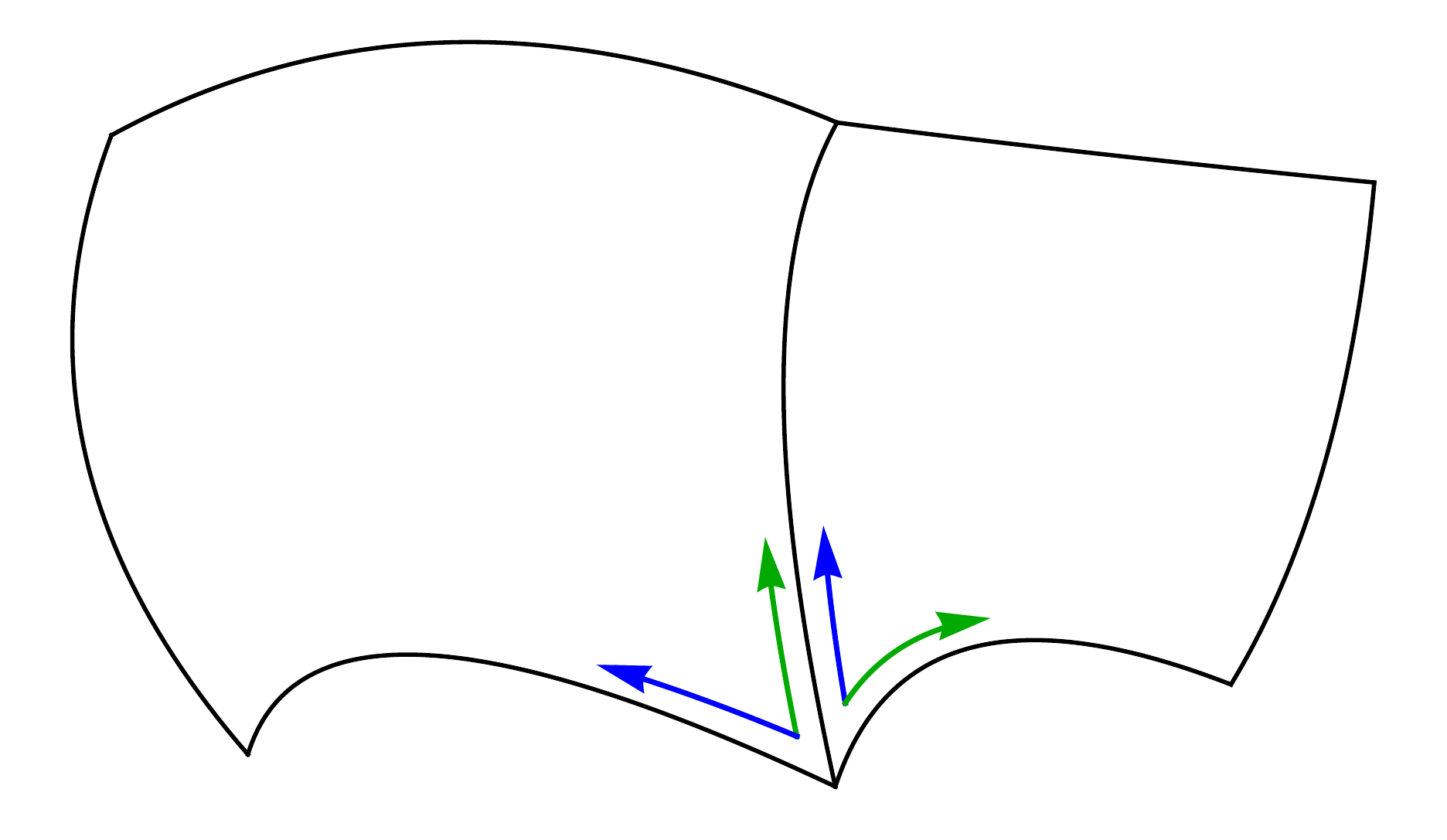}}
    \put(50,55){${\Omega}^{(\ii_2)}$}
    \put(140,55){${\Omega}^{(\ii_1)}$}
    \put(95,80){$\Sigma^{(i)}$}
  \end{picture}
  \hspace{10pt}
  \begin{picture}(160,160)
    \put(0,160){\includegraphics[width=.35\textwidth,angle=270]{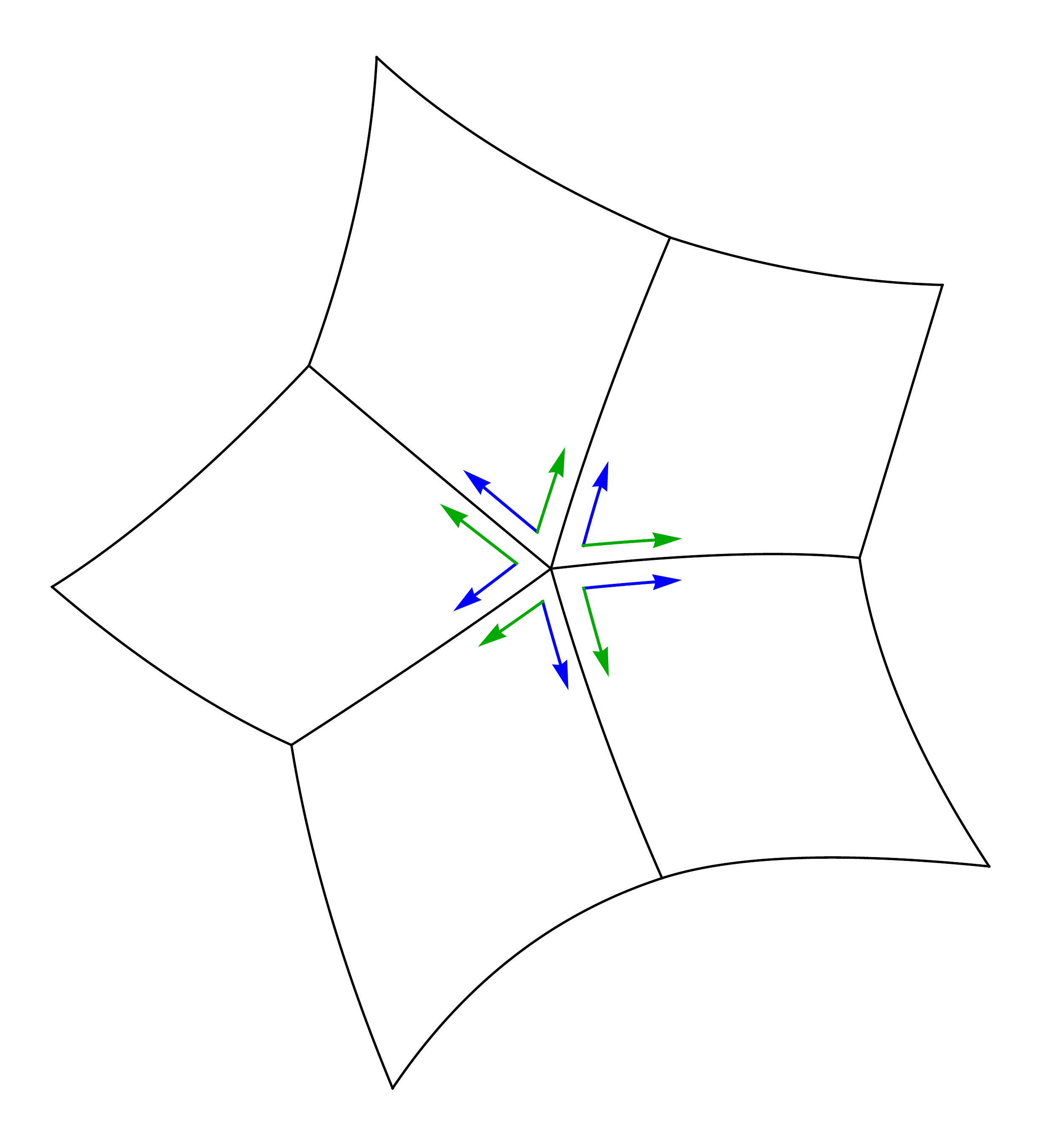}}
    \put(110,40){$\ldots$}
    \put(70,50){$\f x^{(i)}$}
    \put(130,75){${\Omega}^{(\ii_2)}$}
    \put(118,94){$\Sigma^{(\ii_3)}$}
    \put(85,115){${\Omega}^{(\ii_4)}$}
    \put(50,85){$\ldots$}
  \end{picture}
  \caption{Parametrization in standard form for a given
    interface $\Sigma^{(i)}$ (left) and vertex $ \f x^{(i)} $ (right). 
    The local coordinates are depicted in a similar fashion as in Figure~\ref{fig:edge-vertex-numbering}.}
  \label{fig:reparam}
  \label{fig:reparamVertex}
\end{figure}
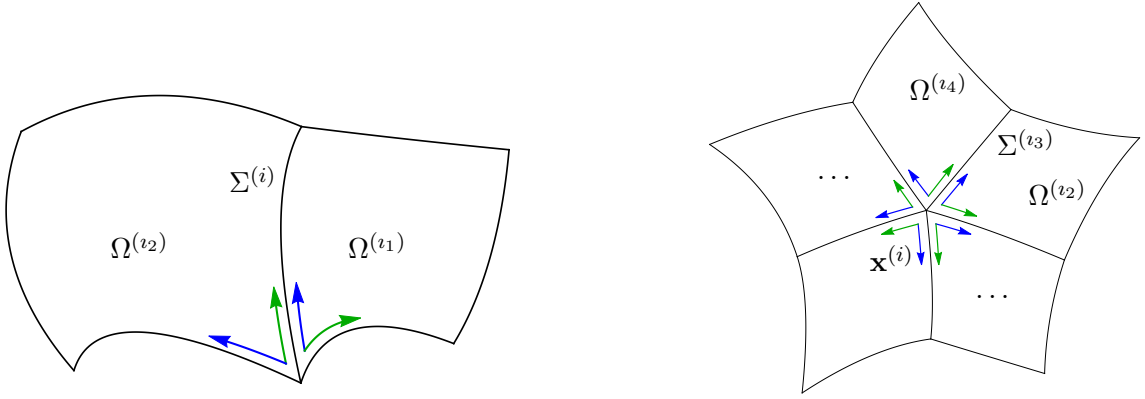

For an interface in standard form, we assume that the two neighboring patches meet in a certain way, as given in the following definition.
\begin{definition}
Given an interface $\Sigma^{(i)}$, for $i \in
\indexSigmaint$, and let $\ii_1$, $\ii_2$ be the corresponding patch indices as in Assumption~\ref{assu:global-local-edge}. We have given a \emph{parametrization in standard
form} for the interface $\Sigma^{(i)}$, if 
\begin{equation}\label{eq:patch-continuity} 
  \f F^{(\ii_{1})} (0,\xi) = \f F^{(\ii_{2})}(\xi,0),\, \text{ for all } \xi\in[0,1].
\end{equation}
This corresponds to a configuration as depicted in Figure~\ref{fig:reparam}~(left). 
Similarly, for a boundary edge $\Sigma^{(i)}$, with $i \in
\indexSigmabound$, we say that a parametrization in standard form is given if $\Sigma^{(i)} = \{\f F^{(\ii_{1})} (0,\xi): \xi\in\left]0,1\right[\}$.
\end{definition}
For a vertex in standard form, we assume that the vertex is enclosed by edges and patches 
$$
  \Sigma^{(\ii_{1})},\;\Omega^{(\ii_{2})},\;\Sigma^{(\ii_{3})},\;\Omega^{(\ii_{4})},\;\ldots,\;\Sigma^{(\ii_{2\nu-1})},\;\Omega^{(\ii_{2\nu})},\;\Sigma^{(\ii_{2\nu+1})}
$$
in counterclockwise order and that the patches are parametrized in such a way 
that the vertex is at the origin, as depicted in Figure~\ref{fig:reparam}~(right). This is detailed in the following definition.
\begin{definition}\label{def:vertex-standard-form}
Given a vertex $\f x^{(i)}$, for $i \in
\indexX$, and let $\ii_2$, \ldots, $\ii_{2\nu}$ be the corresponding patch indices as in Assumption~\ref{assu:global-local-vertex}. 
We have given a \emph{parametrization in standard form} for the vertex $\f x^{(i)}$, if 
\begin{equation}
  \f F^{(\ii_{2\ell})}(0,\xi) = \f F^{(\ii_{2\ell+2})}(\xi,0),\, \text{ for all } \xi \in[0,1],
\label{eq:vertex-reparam}
\end{equation}
where $\ell=1,\ldots,\nu$ for interior vertices while  $\ell=1,\ldots,\nu-1$ for boundary vertices. 
We have  $\overline{\Sigma^{(\ii_{2\ell+1})}} = \overline{\Omega^{(\ii_{2\ell})}} \cap \overline{\Omega^{(\ii_{2\ell+2})}}$, 
with $\ell=0,\ldots,\nu$ for interior vertices, and with $\ell=1,\ldots,\nu-1$ for boundary vertices. For interior vertices we have $\ii_{1} = \ii_{2\nu+1}$, and
for boundary vertices we set instead $\overline{\Sigma^{(\ii_{1})}} = \Gamma \cap \overline{\Omega^{(\ii_{2})}}$
as well as $\overline{\Sigma^{(\ii_{2\nu+1})}} = \overline{\Omega^{(\ii_{2\nu})}} \cap \Gamma$. 
\end{definition}
In the  case  of Definition \ref{def:vertex-standard-form}, all
interfaces are in standard form. The even/odd indexing for
patches/edges will come handy in Section \ref{sec:W}.

When considering a single edge
$\Sigma^{(i)}$ or a single vertex $ \f x^{(i)} $, it is not
restrictive to assume that the given parametrizations are in standard
form, as stated in the following lemma. 
\begin{lemma}[Reparametrizations that yield a standard form]\label{lem:reparam-standard-form}
Let $\f r : [0,1]^2 \rightarrow [0,1]^2$ with $\f
r(\xi_1,\xi_2)=(1-\xi_2,\xi_1)$. 
Given an interface $\Sigma^{(i)}$, for $i \in
\indexSigmaint$, and $\tupleSigma{i}=\{(\ii_1,\kk_1),(\ii_2,\kk_2)\}$, then 
$\f F^{(\ii_{1})}\circ{\f r}^{\kk_1} $  and $\f
F^{(\ii_{2})}\circ{\f r}^{\kk_2-1}$  are  suitable
reparametrizations of  the adjacent  patches $\Omega^{(\ii_{1})}$ and
$\Omega^{(\ii_{2})}$ that yield a parametrization in standard
form for the interface $\Sigma^{(i)}$. Similarly, $\f F^{(\ii_1)}\circ\f r^{\kk_1}$ yields a parametrization in standard form for a boundary edge $\Sigma^{(i)}$ with $i \in\indexSigmabound$. 

Given a vertex $ \f x^{(i)} $, $i \in \indexX$, and $\tupleSigma{i}=\{(\ii_2,\kk_2),\ldots,(\ii_{2\nu},\kk_{2\nu})\}$. Assuming that the patches are arranged in counterclockwise order, then 
$\f F^{(\ii_{2\ell})}\circ{\f r}^{\kk_{2\ell}}$, for $\ell=1,\ldots,\nu$, are  suitable
reparametrizations of  the adjacent  patches $\Omega^{(\ii_{2\ell})}$ that yield a parametrization in standard
form for the vertex $ \f x^{(i)} $.
\end{lemma}
Here, $\f r$ corresponds to the  $\pi/2$  counterclockwise  rotation map in the
parameter domain.  We denote by $\f F^{(\ii)} \circ {\f r}^\kk$ the reparametrization of the patch $\Omega^{(\ii)}$ (with parametrization $\f F^{(\ii)}$) by 
a rotation with an angle of $\kk\pi/2$, see Figure \ref{fig:rot}.
\begin{figure}[ht]
  \centering  \begin{picture}(250,110)
    \put(0,0){\includegraphics[width=.55\textwidth]{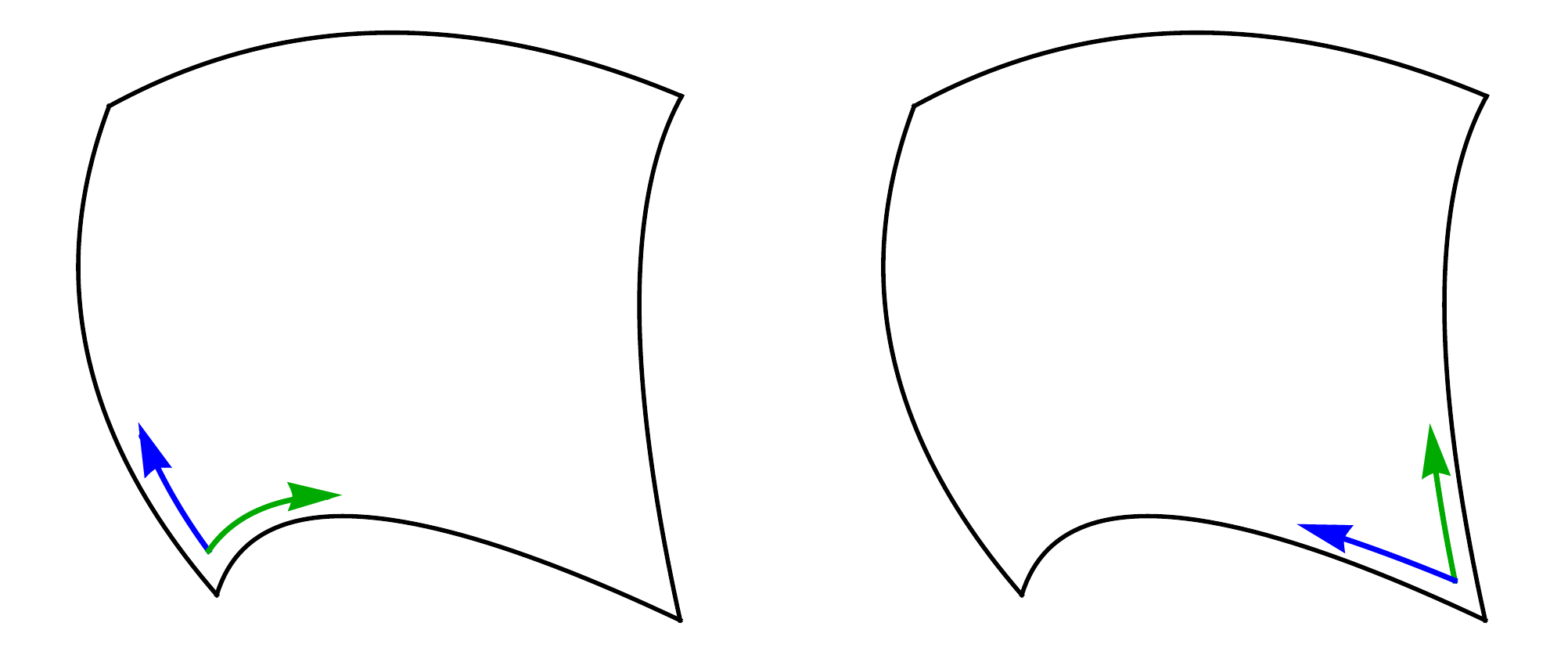}}
    \put(40,55){$\f F^{(\ii)}([0,1]^2)$}
    \put(150,55){$\left(\f F^{(\ii)}\circ {\f r}\right)([0,1]^2)$}
  \end{picture}
  \caption{Visualization of a reparametrization for a given patch $\Omega^{(\ii)}$. 
  The local coordinates are depicted in a similar fashion as in Figure~\ref{fig:edge-vertex-numbering}.}
  \label{fig:rot}
\end{figure}

A proof of Lemma \ref{lem:reparam-standard-form} is straightforward and will be omitted here. 
Actually, for each interface there exist two parametrizations in
standard form,  depending on the order of the indices $\ii_1$ and $\ii_2$, 
(changing the order of the patches simply changes the orientation of the 
interface with respect to the two patches). Similarly, for an interior vertex the choice of the indices 
is not unique; moreover, in case of an interior vertex, all sub-indices are considered to be
modulo $(2\nu)$, e.g., $ \ii_0 = \ii_{2\nu}$.

\section{$C^1$ isogeometric spaces}
\label{sec:c1-iga-spaces}

In this section we define  $C^0$ and
then  $C^1$  multi-patch isogeometric spaces, discuss the  relation to geometric
continuity $G^1$ of the graph parametrization, and the notion of AS-$G^1$ continuity of the
multi-patch parametrization.

\subsection{Isogeometric spaces}
\label{sec:iga-spaces}

\begin{definition}[Isogeometric spaces]\label{defi:V0-V1}
We define the $C^0$ isogeometric space as 
\begin{equation} \label{eq:V0}
 \mathcal{V}^{0}=\left \{ \igf \in  C^0( \overline\Omega ) \, | \, \text{ for all } i \in
   \indexOmega, \, f_h^{(i)} = \igf \circ \f
   F^{(i)} \in \Spr \right\},
  \end{equation}
and the $C^1$ isogeometric space as 
 \begin{equation} \label{eq:V1}
 \mathcal{V}^{1}=\mathcal{V}^{0}\cap C^1( \overline\Omega ).
  \end{equation}
\end{definition}

\subsection{$C^1$ regularity of isogeometric functions and  $G^1$  graph  regularity}
\label{sec:C1-G1-conditions}

The graph  $\graph \subset \Omega \times \RR$ of an isogeometric function  $\igf:\Omega \rightarrow
\mathbb{R}$ is naturally split into patches $\graph^{(i)}$ having the parametrizations 
\begin{equation}\label{eq:isogeom-function-parametrization}
  \left [
    \begin{array}{c}
      \f F^{(i)}\\ f_h^{(i)} 
    \end{array}
\right ]\, : \,  [0,1]^2 \rightarrow \graph^{(i)},
\end{equation}
where $f_h^{(i)} =\igf \circ \f F^{(i)} $. 

The  $C^1$ continuity  at an interface $\Sigma^{(i)}$ is, by
definition, the $G^1$ (geometric) continuity of its graph parametrization, as
stated in the next proposition.
\begin{proposition}\label{teo:C1-G1-graph-parametrization} Under
  Assumption \ref{ass:C1-in-patches},  an isogeometric function $\igf \in \mathcal{V}^{0}$ belongs to 
$\mathcal{V}^1$ if and only if the parametrization 
\eqref{eq:isogeom-function-parametrization} of its graph is $G^1$ continuous  
at the  interfaces $ \overline \Sigma^{(i)} $, for all $i \in \indexSigmaint$.
\end{proposition}
For a discussion and generalizations  of the equivalence above see
\cite{CoSaTa16,Pe15,KaViJu15}. The definition of $G^1$ continuity, in
 our context, is detailed below.
\begin{definition}[$G^1$ continuity at  $\Sigma^{(i)}$]\label{def:G1} 
Consider an interface $\Sigma^{(i)}$, with $i\in\indexSigmaint$.
Assume $\tupleSigma{i}=\{(\ii_1,0),(\ii_2,1)\}$, that is, the adjacent
patches have  parametrizations $\f F^{(\ii_1)}$ and  $\f F^{(\ii_2)}$ in standard
 form for $\Sigma^{(i)}$  (see Section
 \ref{sec:local-parametrization}). 
 The graph parametrization
 \eqref{eq:isogeom-function-parametrization} is said to be 
$G^1$  at $\overline\Sigma^{(i)} $ if there exist functions $\alpha^{(i,\ii_1)}:
\left[0,1\right]   \rightarrow \RR $, $ \alpha^{(i,\ii_2)} : \left[0,1\right] \rightarrow \RR $ and $\beta^{(i)}  : \left[0,1\right]
 \rightarrow \RR $ such that for all $  \xi \in \left[0,1\right]$,
 \begin{equation}
   \label{eq:alpha-sign-condition}
 \alpha^{(i,\ii_1)}  (\xi) \alpha^{(i,\ii_2)}  (\xi) > 0
 \end{equation}
and 
 \begin{equation}
   \label{eq:G1-for-Fg-alpha-beta-gamma}
  	\alpha^{(i,\ii_1)} (\xi)  \left [
    \begin{array}{c}
     \Dv \f F^{(\ii_2)}(\xi,0) \\ \Dv f_h^{(\ii_2)} (\xi,0) 
    \end{array}
\right ]
+
   \alpha^{(i,\ii_2)}(\xi) 
\left [
    \begin{array}{c}
       \Du  \f F^{(\ii_1)}(0,\xi) \\    \Du f_h^{(\ii_1)} (0,\xi) 
    \end{array}
\right ]
+ \beta^{(i)} (\xi)
\left [
    \begin{array}{c}
      \Dv  \f F^{(\ii_1)}(0,\xi) \\ \Dv  f_h^{(\ii_1)}(0,\xi)
    \end{array}
\right ]
=\boldsymbol{0}.
 \end{equation}
\end{definition}
In the framework of Definition \ref{def:G1}, it is useful to introduce functions 
$\beta^{(i,\ii_1)}$ and $\beta^{(i,\ii_2)}$ such that
\begin{equation}
  \label{eq:beta_function_of_alpha}
  \beta^{(i)} (\xi)= \alpha^{(i,\ii_1)} (\xi) \beta^{(i,\ii_2)}(\xi)+ \alpha^{(i,\ii_2)}(\xi)\beta^{(i,\ii_1)}(\xi).
\end{equation}
We call the functions $\alpha^{(i,\ii_1)}$, $\alpha^{(i,\ii_2)}$, 
$\beta^{(i,\ii_1)}$ and $\beta^{(i,\ii_2)}$ the \emph{gluing data} 
 for the interface $\Sigma^{(i)}$. 
 
In the context of IGA, the multi-patch domain parametrizations $ \f F^{(i)}$ are considered given (at
least, at each linearization step)  and determine the gluing data
above, while $f_h^{(i)} $ are  the
unknowns which need to fulfill the gluing condition
(third equation of \eqref{eq:G1-for-Fg-alpha-beta-gamma}) in order to represent $C^1$
isogeometric functions. This is stated in the next result which combines
Proposition~\ref{teo:C1-G1-graph-parametrization} and Definition~\ref{def:G1}.
\begin{proposition}\label{prop:characterization-of-C1}
For each $i\in\indexSigmaint$, assuming
$\tupleSigma{i}=\{(\ii_1,0),(\ii_2,1)\}$ (i.e., $\f F^{(\ii_1)}$ and  $\f
F^{(\ii_2)}$ in standard
 form for $\Sigma^{(i)}$), let  $\alpha^{(i,\ii_1)} $, $
 \alpha^{(i,\ii_2)} $ and $\beta^{(i)}$  such that for all $  \xi \in
 \left[0,1\right]$ it holds $\alpha^{(i,\ii_1)}  (\xi) \alpha^{(i,\ii_2)}
 (\xi) > 0$ and 
 \begin{equation}
   \label{eq:G1-for-F-alone-alpha-beta-gamma}
 	\alpha^{(i,\ii_1)} (\xi)  \Dv \f F^{(\ii_2)}(\xi,0)  +
        \alpha^{(i,\ii_2)}(\xi) \Du  \f F^{(\ii_1)}(0,\xi) + \beta^{(i)} (\xi)
        \Dv  \f F^{(\ii_1)}(0,\xi)  =\boldsymbol{0}.
 \end{equation} 
Let  $\igf \in \mathcal{V}^{0}$. Then $\igf \in \mathcal{V}^{1}$  if
and only if  for all $i\in\indexSigmaint$ and $  \xi \in
 \left[0,1\right]$, it holds
\begin{equation}\label{eq:G1-condition-for-g}
  	\alpha^{(i,\ii_1)} (\xi)  \Dv f_h^{(\ii_2)}(\xi,0)  +
        \alpha^{(i,\ii_2)}(\xi) \Du  f_h^{(\ii_1)}(0,\xi) + \beta^{(i)} (\xi)
        \Dv  f_h^{(\ii_1)}(0,\xi)  =0,
\end{equation}
where $f_h^{(i)} =\igf \circ \f F^{(i)} $.
\end{proposition}
\begin{proof}
 Due to the regularity condition \eqref{eq:non-singular-F}, for each $\xi \in
 \left[0,1\right]$, \eqref{eq:G1-for-F-alone-alpha-beta-gamma} are two
 linearly independent equations for $[\alpha^{(i,\ii_1)} (\xi),
 \alpha^{(i,\ii_2)}(\xi) , \beta^{(i)}(\xi)]$, whose solutions are, for
 completeness, 
 \begin{equation}\label{eq:gluing-data-explicit-formula}
   \begin{aligned}
  	\alpha^{(i,\ii_1)} (\xi)  & = \gamma (\xi)  \det  \left [\begin{array}{ll}
          \Du \f F^{(\ii_1)}(0,\xi) & 
        \Dv \f F^{(\ii_1)}(0,\xi) 
         \end{array}\right ],\\
 	\alpha^{(i,\ii_2)} (\xi)  & =  \gamma (\xi)  \det  \left [\begin{array}{ll}
         \Du \f F^{(\ii_2)}(\xi,0) & 
        \Dv \f F^{(\ii_2)}(\xi,0) 
        \end{array}\right ],\\
 \beta^{(i)} (\xi) & =  \gamma (\xi) \det  \left [\begin{array}{ll}
         \Dv \f F^{(\ii_2)}(\xi,0)  & 
        \Dv \f F^{(\ii_1)}(0,\xi) 
         \end{array}\right ],
   \end{aligned}
  \end{equation}
  where $ \gamma (\xi) $ is arbitrary and we have used $  \Dv  \f
  F^{(\ii_1)}(0,\xi)   =  \Du \f F^{(\ii_2)}(\xi,0) $.

Assume  now $\igf $ is
  $C^1$, that is, \eqref{eq:G1-for-Fg-alpha-beta-gamma} holds,  by Proposition~\ref{teo:C1-G1-graph-parametrization}. 
 The gluing data allowed in \eqref{eq:G1-for-Fg-alpha-beta-gamma}
are already determined by \eqref{eq:G1-for-F-alone-alpha-beta-gamma}, as given in \eqref{eq:gluing-data-explicit-formula}. The third equation 
of \eqref{eq:G1-for-Fg-alpha-beta-gamma}  is the same as  \eqref{eq:G1-condition-for-g}.

Conversely, if \eqref{eq:G1-condition-for-g}
  holds, together with \eqref{eq:G1-for-F-alone-alpha-beta-gamma} 
  it yields    \eqref{eq:G1-for-Fg-alpha-beta-gamma}, that is,  $\igf $ is
  $C^1$.
\end{proof}
 What this proposition shows, is that the gluing data are completely determined (up to a common factor $\gamma(\xi)$) by 
the patch parametrizations $\f F^{(\ii_1)}$, $\f F^{(\ii_2)}$. The $C^1$ condition on the isogeometric function $\varphi_h$ is then a linear constraint on the functions 
$f_h^{(\ii_1)}$, $f_h^{(\ii_2)}$ in the parameter domain. 
The proof above also shows that one can select the gluing
data such that   $\alpha^{(i,\ii_1)},\alpha^{(i,\ii_2)} \in \mathcal{S}^{2p-1,r-1}_h $ and $\beta^{(i)} \in \mathcal{S}^{2p,r}_h $.
A special case is when the gluing data are polynomial
functions. This happens in particular (but not only) for  B\'ezier patches, which is often the case
in literature. In this situation, since   $\alpha^{(i,\ii_1)}$,
$\alpha^{(i,\ii_2)}$ and $\beta^{(i)}$  are determined up to a common
multiplicative function, it is not restrictive to assume the
following.
\begin{assumption}[Simplification of gluing data]
{If $\alpha^{(i,\ii_1)}$ and $\alpha^{(i,\ii_2)}$ are polynomial functions, we assume that they are} relatively
prime (i.e.,
$\deg(\gcd(\alpha^{(i,\ii_1)},\alpha^{(i,\ii_2)}))=0$).
\end{assumption}
Note also that the choice of $\beta^{(i,\ii_1)}$ and $\beta^{(i,\ii_2)}$
is not unique.\footnote{Even if not necessary in this paper, from
  the practical point of view it is advisable to select stable  gluing data e.g. by minimizing 
$\|\alpha^{(i,\ii_1)}-1\|^2_{L^2(0,1)} + \|\alpha^{(i,\ii_2)}-1\|^2_{L^2(0,1)}$ as well as $\|\beta^{(i,\ii_1)}\|^2_{L^2(0,1)} + \|\beta^{(i,\ii_2)}\|^2_{L^2(0,1)}$. In case of  parametric continuity, i.e.,
$\beta^{(i)}\equiv 0$ and $\alpha^{(i,\ii_1)}=\alpha^{(i,\ii_2)}$, this implies $\beta^{(i,\ii_1)}\equiv\beta^{(i,\ii_2)}\equiv 0$ and $\alpha^{(i,\ii_1)}=\alpha^{(i,\ii_2)}\equiv 1$.}
One can show that there exist piecewise polynomial functions
$\beta^{(i,\ii_1)}$ and $\beta^{(i,\ii_2)}$ that satisfy equation
\eqref{eq:beta_function_of_alpha}.

\subsection{Analysis-suitable $G^1$ condition}
\label{sec:ASC1-isogeometric-spaces}

In order to ensure {optimal reproduction properties 
for the trace and normal derivative along the interfaces of the isogeometric space $\mathcal{V}^{1}$}, we introduce an additional condition  on the geometry
parametrization, as in \cite{CoSaTa16,KaSaTa17}, stated as a condition
on the gluing data $\alpha^{(i,\ii_1)}$, $\alpha^{(i,\ii_2)}$, $\beta^{(i,\ii_1)}$ and $\beta^{(i,\ii_2)}$.

\begin{definition}[Analysis-suitable $G^1$
  parametrization] \label{def:AS-G1} The  parametrizations $\f F^{(\ii_1)}
  $ and  $\f F^{(\ii_2)} $ are \emph{analysis-suitable $G^1$}  at the interface $ \Sigma^{(i)}$ (in short,
  AS-$G^1(\Sigma^{(i)})$ or  AS-$G^1$)  if there exist  gluing data $\alpha^{(i,\ii_1)}$, $\alpha^{(i,\ii_2)}$,
$\beta^{(i,\ii_1)}$ and $\beta^{(i,\ii_2)}$ that are linear polynomials and
such that \eqref{eq:G1-for-Fg-alpha-beta-gamma}--(\ref{eq:beta_function_of_alpha}) hold.
\end{definition}
The class of AS-$G^1$ parametrizations contains, but is not restricted to,  
bilinear parametrizations, see \cite{CoSaTa16,KaSaTa17b}. 
Note that this is a non-trivial requirement. However, it was 
shown in \cite{KaSaTa17b} that many multi-patch geometries can be 
reparametrized to satisfy the AS-$G^1$ constraints, motivating the
following assumption.
\begin{assumption}
We assume that for all interfaces $ \Sigma^{(i)}$, $i\in\indexSigmaint$, the parametrizations 
$\f F^{(\ii_1)}$ and  $\f F^{(\ii_2)} $ are analysis-suitable $G^1$.
\end{assumption}

\section{The Argyris isogeometric space $\W\subset \mathcal{V}^1$}
\label{sec:W}

Unlike $\mathcal{V}^0$, the isogeometric space $\mathcal{V}^1$ has a
complex structure. Every interface generates certain $C^1$ constraints, 
which are usually not independent of each other. 
At every interior vertex, several interfaces meet, which may lead to possibly conflicting constraints. 
The geometry needs to satisfy additional conditions there. 
In the context of computer-aided geometric design these are sometimes called
vertex enclosure constraints, see e.g. \cite{HoLa93}.  
In addition to this issue concerning vertices, 
the dimension of
 $\mathcal{V}^1$  depends also on the
 domain parametrization even for the simplest configuration of 
 two bilinear patches, as shown in \cite{KaSaTa17,KaViJu15}. For this reason,  instead of dealing with
 $\mathcal{V}^1$ itself, we introduce  a suitable
 subspace $\W \subset \mathcal{V}^1$ which is simpler to construct and
 has a dimension which is independent of the geometry
 parametrization. The space $\W$ is named Argyris isogeometric space
 since it represents an extension of the classical Argyris finite
 element space, as explained in the Introduction.

\subsection{Description and properties of $\W$}

The functions spanning the space $\W$ are standard isogeometric
functions within the patches, possess a special structure at the patch
interfaces and edges (which is motivated by \cite{CoSaTa16}) and are
$C^2$ at all vertices.

Precisely, we define $\W $  as the (direct) sum of interior, edge and vertex components: 
\begin{equation}
  \label{eq:W}
\W = 	\left(\bigoplus_{i\in\indexOmega}\W^{\circ}_{\Omega^{(i)}}\right)
	\oplus\left(\bigoplus_{i\in\indexSigma}\W^{\circ}_{\Sigma^{(i)}}\right)
	\oplus\left(\bigoplus_{i\in\indexX}\W_{{\f x}^{(i)}}\right).
\end{equation}
The patch-interior basis functions spanning $\W^{\circ}_{\Omega^{(i)}}$ 
have their support entirely contained in one patch. They are taken as those functions supported on the patch $\Omega^{(i)}$ 
which have vanishing function values and gradients at the entire boundary of the patch.

The edge-interior basis functions spanning $\W^{\circ}_{\Sigma^{(i)}}$ have their support 
contained in an ($h$-depen\-dent) neighborhood of the edge $\Sigma^{(i)}$. They span function values and normal derivatives 
along the interface and have vanishing derivatives up to second order at the endpoints of the interface (vertices of the multi-patch domain). 
They are thus supported in at most two patches.

The vertex basis functions spanning $\W_{{\f x}^{(i)}}$ are supported 
within an ($h$-dependent) neighborhood of the vertex ${\f x}^{(i)}$. There are exactly six such functions per vertex and they span 
the function value and all derivatives up to second order at the vertex. Hence, they are $C^2$ at the vertex by definition, see Proposition \ref{prop:smoothness-of-W}. 

The precise definitions of the different types of basis functions will be given in the three sections below, more precisely in 
Definitions \ref{lem:patch-space}, \ref{def:edge-space} and \ref{def:vertex-basis}, for the patch-interior, edge-interior and vertex basis, respectively.

Note that for splines of maximal smoothness $r=p-1$ within the patches
there is no convergence of the approximation error 
under $h$-refinement even on a simple bilinear two-patch domain, see \cite{CoSaTa16}. 
Hence, we have the following request.
\begin{assumption}[Maximal regularity within the patches]\label{ass:C(p-2)-in-patches}
  We assume $r\leq p-2$.
\end{assumption}
The assumption above is needed to  allow $h$-refinement (see
  \cite{CoSaTa16}).
Moreover, the  split \eqref{eq:W} itself
is well defined  if the spline spaces are sufficiently
refined, as stated in the next assumption.
\begin{assumption}[Minimal mesh resolution within the patches]\label{ass:minimal-mesh-within-patches}
  We assume $\mathcal{S}^{\f p,\f r}_{h} $ has  $ n \geq \frac{4-r}{p-r-1}$ elements per direction, that is  $h\leq
  \frac{p-r-1}{4-r}$. 
\end{assumption}

\begin{remark} A B\'ezier patch (i.e., $h=1$) fulfills Assumption
  \ref{ass:minimal-mesh-within-patches} and, trivially, Assumption
  \ref{ass:C(p-2)-in-patches},   for degree $p \geq 5$. This is the 
  case considered in  \cite{BeMa14,mourrain2015geometrically}. However, we study
  the subspace $\A \subset \mathcal{V}^{1}$.
\end{remark}

\begin{remark}
The dimensions of the subspaces in \eqref{eq:W} do not depend on the geometry and satisfy
\begin{equation*}
    \dim(\W^{\circ}_{\Omega^{(i)}}) = \left((p-r)(n-1)+p-3\right)^2,
\end{equation*}
for each $i\in\indexOmega$,
\begin{equation*}
    \dim(\W^{\circ}_{\Sigma^{(i)}}) = 2(p-r-1)(n-1)+p-9,
\end{equation*}
for each $i\in\indexSigma$, as well as 
\begin{equation*}
    \dim(\W_{{\f x}^{(i)}}) = 6,
\end{equation*}
for each $i\in\indexX$, cf. Definition~\ref{lem:patch-space}, \ref{def:edge-space}, \ref{def:boundary-edge-space} and Lemma~\ref{lem:full-vertex-proj-physical-intrpolation}.
Hence, the dimension of $\W$ is completely determined by the degree~$p$, the regularity~$r$, the number of elements in each direction~$n$, as well as the number of patches, edges and vertices, via
\begin{equation*}
    \dim(\W) = |\indexOmega| \cdot \left((p-r)(n-1)+p-3\right)^2 + |\indexSigma| \cdot \left(2(p-r-1)(n-1)+p-9\right) + |\indexX| \cdot 6.
\end{equation*}
\end{remark}

\subsection{The patch-interior function space $\W^{\circ}_{\Omega^{(i)}}$}
\label{sec:patch-inter-space}

For each patch $\Omega^{(i)}$ with $i\in\indexOmega$, we define a function space $\W^{\circ}_{\Omega^{(i)}}$ as the span of all basis functions supported in $\Omega^{(i)}$, which have vanishing function value and vanishing gradients at the patch boundary $\partial\Omega^{(i)}$. For the following definition, recall that $b_{\f j}$, for $\f j \in\mathbb{I}$, is the standard tensor-product B-spline basis for the spline space $\Spr$.

\begin{definition}[Patch-interior basis]\label{lem:patch-space}
Let $i\in\indexOmega$, then we define
\begin{equation}\label{eq:W-Omega-i}
  \W_{\Omega^{(i)}} = \text{span}\,\left\{\basisfct{i}{\f j}:     \overline \Omega^{(i)}  \rightarrow\RR\text{ such that }
    \basisfct{i}{\f j}\circ \f F^{(i)} = b_{\f j}, \mbox{ for } \f j
    \in \mathbb{I}\right\} 
\end{equation}
and 
\begin{equation}\label{eq:W-Omega-i-0}
  \W^{\circ}_{\Omega^{(i)}} = \text{span}\,\left\{\basisfct{i}{\f j}:
     \overline \Omega^{(i)}  \rightarrow\RR\text{ such that }
    \basisfct{i}{\f j}\circ \f F^{(i)} = b_{\f j}, \mbox{ for } \f j \in
\mathbb{I}^{\circ}_{\Omega^{(i)}} \right\} 
\end{equation}
with $   \mathbb{I}^{\circ}_{\Omega^{(i)}} =  \{2,\ldots,N-3\}^2
\subset \mathbb{I} $.
\end{definition}
With a little abuse of notation,  we consider the functions of  $
\W^{\circ}_{\Omega^{(i)}}  $ as defined on the whole $ \overline
\Omega$ by extending to zero outside $\overline \Omega^{(i)}$. We
easily have then $\W^{\circ}_{\Omega^{(i)}}
\subseteq \V^1$ and in particular 
\begin{displaymath}
  \W^{\circ}_{\Omega^{(i)}} = \left \{ \igf  \in \V^1 \text{ such that } \igf(\f x)= 0
    \text{ and } \nabla \igf(\f x)= \f 0, \, \text{ for all } \f x \in
    \overline \Omega \setminus  { \Omega^{(i)}} \right \}.
\end{displaymath}

We also define straightforwardly a projection operator onto the subspace $\W^{\circ}_{\Omega^{(i)}}$.
\begin{definition}[Patch-interior dual basis and projector]\label{def:patch-dual-basis}
Let $\{\lambda_{\f j} = \lambda_{j_1}\otimes\lambda_{j_2}\}_{\f j\in\mathbb{I}}$ 
be a dual basis for the basis $\{b_{\f j}\}_{\f j\in\mathbb{I}}$. For
each $i\in\indexOmega$, we define the projector
$\Pi_{\W^{\circ}_{\Omega^{(i)}}} : L^2\left(\Omega\right)\rightarrow
\mathcal{V}^1$ such that  
\begin{displaymath}
  \Pi_{\W^{\circ}_{\Omega^{(i)}}} (\varphi )= \sum_{\f j\in
  {\mathbb{I}^{\circ}_{\Omega^{(i)}}}} \Lambda_{\f j}(\varphi ) \basisfct{i}{\f j},
\end{displaymath}
where $ \Lambda_{\f j}(\varphi)= \lambda_{\f j}(\varphi\circ \f F^{(i)})$.
\end{definition}

\subsection{The edge function space $\W^{\circ}_{\Sigma^{(i)}}$}
\label{sec:edge-space}

In this section, we consider  first    the most interesting case of interior edges, that is
when $\Sigma^{(i)}$, for $i\in\indexSigmaint$, is an interface between the patches
$\Omega^{(\ii_1)}$ and $\Omega^{(\ii_2)}$. The extension to boundary edges is straightforward and 
will be discussed briefly after Definition \ref{def:edge-space}. The parametrizations  $\f
F^{(\ii_1)}$, $\f F^{(\ii_2)}$  are assumed to be in standard 
 form for $\Sigma^{(i)}$, see Section \ref{sec:local-parametrization}. 
The first step is the definition of a space over $ \overline \Omega^{(\ii_1)}\cup \overline
\Omega^{(\ii_2)}$, which   contains isogeometric
functions fulfilling the $C^1$ constraints at the interface
$\Sigma^{(i)}$.  We define $\W_{\Sigma^{(i)}}$ as the direct sum of 
two subspaces $\W_{\Sigma^{(i)},0}$ and $\W_{\Sigma^{(i)},1}$,
following the construction  in \cite{CoSaTa16}. The spaces 
span the function values and the cross derivative values along the
interface, respectively\footnote{The complete $C^1$  space is
  slightly larger for certain configurations, see \ref{sec:another_space} and \cite{KaSaTa17} for
  a construction of the complete basis.}.
The number of
elements in the  spaces  $\mathcal{S}^{p,r+1}_h $ and
$\mathcal{S}^{p-1,r}_h $  considered
below is the same  and is denoted by $n$.
\begin{definition}[Basis at the interfaces]\label{def:edge-space}
Let $\Sigma^{(i)}$, for $i\in\indexSigmaint$, be an interface in standard form. 
Consider  the univariate spline spaces $\mathcal{S}^+ = \mathcal{S}^{p,r+1}_h$
and $\mathcal{S}^- = \mathcal{S}^{p-1,r}_h$, with bases $\{ {b}^+_j\}_{j \in\mathbb{I}^+}$, and $\{{b}^-_j\}_{j
  \in\mathbb{I}^-}$, respectively, where 
$\mathbb{I}^\pm=\{0,\ldots,N^\pm-1\}$ with  $N^-=(p-r-1)(n-1)+p$ and
$N^+= N^-+1= (p-r-1)(n-1)+p+1$.
Recall that $\{ b_j\}_{j \in  \{0,1,\ldots,N-1\}}$ is the standard univariate B-spline basis 
for $\mathcal{S}^{p,r}_h$, where  $N=(p-r)(n-1)+p+1$.

 We define 
\begin{equation}
  \W_{\Sigma^{(i)},0} = \mbox{span}\, \left\{\basisfctSigma{i}{(j_1,0)}: \overline \Omega^{(\ii_1)}\cup \overline
\Omega^{(\ii_2)}  \rightarrow \mathbb{R}, \text { for } j_1 \in \mathbb{I}^+\right\},
\end{equation}
with
\begin{equation}\label{eq:bj-edge-trace}
  \begin{aligned}
    \basisfctSigma{i}{(j_1,0)}\circ \f F^{(\ii_1)}(\xi_1,\xi_2)=\bar{f}^{(i,\ii_1)}_{(j_1,0)}&=
 {b}^+_{j_1}(\xi_2)\ibasis_0(\xi_1) -  \beta^{(i,\ii_1)}(\xi_2) ({b}^+_{j_1})'(\xi_2) \ibasis_1(\xi_1) ,\\
   \basisfctSigma{i}{(j_1,0)}\circ \f F^{(\ii_2)}(\xi_1,\xi_2)=\bar{f}^{(i,\ii_2)}_{(j_1,0)}&=
 {b}^+_{j_1}(\xi_1)\ibasis_0(\xi_2) -  \beta^{(i,\ii_2)}(\xi_1) ({b}^+_{j_1})'(\xi_1) \ibasis_1(\xi_2) ,
  \end{aligned}
\end{equation}
where $\ibasis_0 = b_0+b_1$ and $\ibasis_1 = \frac{h}{p} b_1$, and 
\begin{equation}
  \W_{\Sigma^{(i)},1} = \mbox{span}\, \left\{\basisfctSigma{i}{(j_1,1)}: \overline \Omega^{(\ii_1)}\cup \overline
\Omega^{(\ii_2)}  \rightarrow \mathbb{R}, \text { for } j_1 \in \mathbb{I}^-\right\},
\end{equation}
with
\begin{equation}\label{eq:bj-edge-derivative}
  \begin{aligned}
    \basisfctSigma{i}{(j_1,1)}\circ \f F^{(\ii_1)}(\xi_1,\xi_2)=\bar{f}^{(i,\ii_1)}_{(j_1,1)}&=\alpha^{(i,\ii_1)}(\xi_2) {b}^-_{j_1}(\xi_2) { b_1}(\xi_1) ,\\
   \basisfctSigma{i}{(j_1,1)}\circ \f F^{(\ii_2)}(\xi_1,\xi_2)=\bar{f}^{(i,\ii_2)}_{(j_1,1)}&= - \alpha^{(i,\ii_2)}(\xi_1) {b}^-_{j_1}(\xi_1) { b_1}(\xi_2) .
  \end{aligned}
\end{equation}

We define 
\begin{equation}\label{eq:WSigma}
  \W_{\Sigma^{(i)}} = \W_{\Sigma^{(i)},0} \oplus \W_{\Sigma^{(i)},1} =
  \mbox{span}\, \left\{\basisfctSigma{i}{\f j}: \overline \Omega^{(\ii_1)}\cup \overline
\Omega^{(\ii_2)}  \rightarrow \mathbb{R}, \text { for } \f j \in \mathbb{I}_{\Sigma^{(i)}}\right\},
\end{equation}
with $\mathbb{I}_{\Sigma^{(i)}} = \left(\mathbb{I}^+ \times \{0\}\right) \cup \left(\mathbb{I}^- \times \{1\}\right)$.

Finally we  define the space $\W^{\circ}_{\Sigma^{(i)}}$ as 
\begin{equation}\label{eq:WSigma0}
  \W^{\circ}_{\Sigma^{(i)}} = 
  \mbox{span}\, \left\{\basisfctSigma{i}{\f j}: \overline \Omega^{(\ii_1)}\cup \overline
\Omega^{(\ii_2)}  \rightarrow \mathbb{R}, \text { for }\f j \in \mathbb{I}^{\circ}_{\Sigma^{(i)}}\right\},
\end{equation}
with $\mathbb{I}^{\circ}_{\Sigma^{(i)}} = \left\{ \f j \in
    \mathbb{I}_{\Sigma^{(i)}}:  j_1 + j_2 \geq 3 \text{ and }  j_1
    \leq N^--3 \right\}$.
\end{definition}

\begin{remark}
Assumption  \ref{ass:C(p-2)-in-patches}  guarantees that   $\mathcal{S}^+$ and  $\mathcal{S}^-$ are   proper spline
spaces, that is, piecewise polynomials when $h < 1$  (see also
Theorem 1 of~\cite{CoSaTa16}). Assumption
\ref{ass:minimal-mesh-within-patches} guarantees that  
the space $\W ^{\circ}_{\Sigma^{(i)}}$ is nonempty  
($N^-\geq 5$).
\end{remark}

For completeness, we give now the definition of the basis at the
boundary edge, which is just a simplification of the previous one.

\begin{definition}[Basis at the boundary edges]\label{def:boundary-edge-space}
Let $\Sigma^{(i)}$, for $i\in\indexSigmabound$, be a boundary edge in
standard form.  With the same notation as in Definition
\ref{def:edge-space}, we define 
\begin{equation}
  \W_{\Sigma^{(i)},0} = \mbox{span}\, \left\{\basisfctSigma{i}{(j_1,0)}: \overline \Omega^{(\ii_1)} \rightarrow \mathbb{R}, \text { for } j_1 \in \mathbb{I}^+\right\},
\end{equation}
with
\begin{equation}\label{eq:boundary-bj-edge-trace}
  \begin{aligned}
    \basisfctSigma{i}{(j_1,0)}\circ \f F^{(\ii_1)}(\xi_1,\xi_2)=\bar{f}^{(i,\ii_1)}_{(j_1,0)}&=
 {b}^+_{j_1}(\xi_2)\ibasis_0(\xi_1) ,
  \end{aligned}
\end{equation}
and 
\begin{equation}
  \W_{\Sigma^{(i)},1} = \mbox{span}\, \left\{\basisfctSigma{i}{(j_1,1)}: \overline \Omega^{(\ii_1)}\rightarrow \mathbb{R}, \text { for } j_1 \in \mathbb{I}^-\right\},
\end{equation}
with
\begin{equation}\label{eq:boundary-bj-edge-derivative}
  \begin{aligned}
    \basisfctSigma{i}{(j_1,1)}\circ \f F^{(\ii_1)}(\xi_1,\xi_2)=\bar{f}^{(i,\ii_1)}_{(j_1,1)}&={b}^-_{j_1}(\xi_2) { b_1}(\xi_1) .\\
  \end{aligned}
\end{equation}
We set 
\begin{equation}\label{eq:boundary-WSigma}
  \W_{\Sigma^{(i)}} = \W_{\Sigma^{(i)},0} \oplus \W_{\Sigma^{(i)},1} =
  \mbox{span}\, \left\{\basisfctSigma{i}{\f j}: \overline \Omega^{(\ii_1)} \rightarrow \mathbb{R}, \text { for } \f j \in \mathbb{I}_{\Sigma^{(i)}}\right\},
\end{equation}
with $\mathbb{I}_{\Sigma^{(i)}} = \left(\mathbb{I}^+ \times
  \{0\}\right) \cup \left(\mathbb{I}^- \times \{1\}\right)$, and 
\begin{equation}\label{eq:boundary-WSigma0}
  \W^{\circ}_{\Sigma^{(i)}} = 
  \mbox{span}\, \left\{\basisfctSigma{i}{\f j}: \overline \Omega^{(\ii_1)}\rightarrow \mathbb{R}, \text { for }\f j \in \mathbb{I}^{\circ}_{\Sigma^{(i)}}\right\},
\end{equation}
with $\mathbb{I}^{\circ}_{\Sigma^{(i)}} = \left\{ \f j \in
    \mathbb{I}_{\Sigma^{(i)}}:  j_1 + j_2 \geq 3 \text{ and }  j_1
    \leq N^--3 \right\}$.
\end{definition}

As for the patch-interior space, we can extend the functions
of  $\W_{\Sigma^{(i)}}$ and $\W^{\circ}_{\Sigma^{(i)}}$ to zero. Remarkably, we have the
following two  inclusions.
\begin{lemma}\label{lemma:edge-fun-1}
 We have  
\begin{equation*}
  \W_{\Sigma^{(i)}} \subset C^1(\Sigma^{(i)}). 
\end{equation*}
\end{lemma}
\begin{proof}
 By construction, the pair of functions \eqref{eq:bj-edge-trace}
 fulfills condition \eqref{eq:G1-condition-for-g}, that is for all $\xi
 \in [0,1]$, 
 \begin{displaymath}
   	\alpha^{(i,\ii_1)} (\xi)  \Dv \bar{f}^{(i,\ii_2)}_{(j_1,0)} (\xi,0)  +
        \alpha^{(i,\ii_2)}(\xi) \Du  \bar{f}^{(i,\ii_1)}_{(j_1,0)} (0,\xi) + \beta^{(i)} (\xi)
        \Dv  \bar{f}^{(i,\ii_1)}_{(j_1,0)} (0,\xi)  =0,
 \end{displaymath}
 The same holds for
 the pair of functions \eqref{eq:bj-edge-derivative}: For all $\xi
 \in [0,1]$, 
 \begin{displaymath}
   	\alpha^{(i,\ii_1)} (\xi)  \Dv \bar{f}^{(i,\ii_2)}_{(j_1,1)} (\xi,0)  +
        \alpha^{(i,\ii_2)}(\xi) \Du  \bar{f}^{(i,\ii_1)}_{(j_1,1)} (0,\xi) + \beta^{(i)} (\xi)
        \Dv  \bar{f}^{(i,\ii_1)}_{(j_1,1)} (0,\xi)  =0.
 \end{displaymath}
The statement follows thanks to
 Proposition~\ref{prop:characterization-of-C1}, see \cite{CoSaTa16}
 for more details. 
\end{proof}
\begin{proposition}\label{lemma:edge-fun-2}
 We have 
\begin{equation*}
  \W^{\circ}_{\Sigma^{(i)}} \subset \V^1.
\end{equation*}
\end{proposition}
\begin{proof}
One can easily  characterize
$\W^{\circ}_{\Sigma^{(i)}}$ as the subset of functions of
$\W_{\Sigma^{(i)}}$ having null value, gradient and Hessian at the
edge endpoints. Then, by construction, the functions in
$\W^{\circ}_{\Sigma^{(i)}}$ have vanishing trace and derivative at
$\Sigma^{(i')}$, for all $i' \in \indexSigma$, $i' \neq i$. The statement follows.
\end{proof}

We further define a dual basis and projection operator  onto $\W^{\circ}_{\Sigma^{(i)}}$.
\begin{definition}[Edge-interior dual basis and projector]\label{def:edge-space-dual-proj}
Let $\{\lambda^+_{j}\}_{j\in\II^+}$ be  a dual basis for
$\{b^+_{j}\}_{j\in\II^+}$ and $\{\lambda^-_{j}\}_{j\in\II^-}$  a dual
basis for $\{b^-_{j}\}_{j\in\II^-}$. We define  
\begin{equation*}
  \overline{\Lambda}^{(i)}_{(j_1,0)}(\varphi) = \lambda^+_{j_1}\left(\varphi\circ\f F^{(\ii_1)}(0,\bullet)\right), 
\end{equation*}
\begin{equation*}
  \overline{\Lambda}^{(i)}_{(j_1,1)}(\varphi) = \lambda^-_{j_1}\left({\frac{h}{p}}(\nabla\varphi)\circ\f F^{(\ii_1)}(0,\bullet) \cdot {\fn}^{(i)}(\bullet)\right),
\end{equation*}
where  $\fn^{(i)}$  is the transversal vector 
$$
  \fn^{(i)}(\xi) = \frac{1}{\alpha^{(i,\ii_{1})}(\xi)}\left( \Du \f F^{(\ii_{1})}(0,\xi) + \beta^{(i,\ii_{1})}(\xi) \, \Dv \f F^{(\ii_{1})}(0,\xi)\right).
$$
We define the projector $ \Pi_{\W^{\circ}_{\Sigma^{(i)}}} :
C^1\left( \Sigma^{(i)} \right) \rightarrow \W^{\circ}_{\Sigma^{(i)}} $  such that
\begin{equation}
  \label{eq:projector-on-edge-space}
  \Pi_{\W^{\circ}_{\Sigma^{(i)}}} ( \varphi) = \sum_{\f j\in {\mathbb{I}^{\circ}_{\Sigma^{(i)}}}} \overline{\Lambda}^{(i)}_{\f j}(\varphi) \basisfctSigma{i}{\f j}.
\end{equation}
\end{definition}
\begin{remark}
The projector $ \Pi_{\W^{\circ}_{\Sigma^{(i)}}}$ inherits its properties (such as the locality of the support) 
from the basis functions $\basisfctSigma{i}{\f j}$ 
and from the univariate dual bases $\{\lambda^+_{j}\}_{j\in\II^+}$ and $\{\lambda^-_{j}\}_{j\in\II^-}$.
\end{remark}
\begin{remark} One can define  $\Pi_{\W^{\circ}_{\Sigma^{(i)}}} $
  beyond $C^1\left( \Sigma^{(i)} \right)$, e.g. $H^{3/2 +
    \epsilon}(\Omega)$, for $\epsilon > 0$, suffices.
\end{remark}

\subsection{The vertex function space $\W_{{\f x}^{(i)}}$}
\label{sec:vertex_space}

Let  $i \in \indexX$, and ${\f x}^{(i)}$  be a  vertex with $\Sigma^{(\ii_1)}$,   $\Omega^{(\ii_2)}$, $\Sigma^{(\ii_3)}$, \ldots,
$\Omega^{(\ii_{2\nu})}$, $\Sigma^{(\ii_{2\nu+1} )}$  the sequence of edges and patches 
around ${\f x}^{(i)}$ in counterclockwise order. Throughout this
section, we always assume that the parametrizations $\f F^{(\ii_2)} $, \ldots,  $\f F^{(\ii_{2\nu})} $ are in
standard form for the vertex  ${\f x}^{(i)}$,   as stated in Section \ref{sec:local-parametrization}.

We define the vertex function space $\W_{{\f x}^{(i)}}$ via a suitable projection operator.
\begin{lemma}\label{lem:full-vertex-proj-physical-intrpolation}
There exists a function space $\W_{{\f x}^{(i)}} \subset \V^1$ of dimension $6$ and a suitable projector $\Pi_{\W_{{\f x}^{(i)}}} : C^2({\f x}^{(i)}) \rightarrow \W_{{\f x}^{(i)}} \subset \V^1$, such that for all $\varphi \in C^2({\f x}^{(i)}) $ it holds 
\begin{equation}\label{eq:vertex_interpolation_sum} 
  \partial_{x_1}^{m_1} \partial_{x_2}^{m_2} \, ( \Pi_{\W_{{\f
        x}^{(i)}}} \varphi) ({\f x}^{(i)}) =  \partial_{x_1}^{m_1} \partial_{x_2}^{m_2} \, \varphi ({\f x}^{(i)}) 
\end{equation}
for $ m_1, m_2\geq 0 $ and $m_1+m_2\leq 2$.
\end{lemma}
We will give a constructive proof of this lemma later. To do so, we need to establish some preliminary constructions. Having given such an interpolatory projector, we can now define a basis for the vertex space $\W_{{\f x}^{(i)}}$.
\begin{definition}[Basis at the vertices]\label{def:vertex-basis}
Let $\W_{{\f x}^{(i)}}$ and $\Pi_{\W_{{\f x}^{(i)}}} $ be as in Lemma \ref{lem:full-vertex-proj-physical-intrpolation}. Let
  \begin{displaymath}
    \sigma = \left (  \frac{h}{ p \, \nu} \sum_{\ell = 1}^\nu \|\nabla \f
      F^{(\ii_{2\ell})} (0,0) \| \right ) ^{-1};
  \end{displaymath}
then
\begin{equation}\label{eq:vertex-basis-representation}
  \W_{\f x^{(i)}} = \mbox{span}\, \left\{ \basisfctX{i}{\f j}:\f j \in \mathbb{I}_{\f x^{(i)}}\right\} \subset \V^1,
\end{equation}
where $\basisfctX{i}{\f j} $
\begin{equation}\label{eq:basisfctX-defn}
	\basisfctX{i}{\f j} = \Pi_{\W_{{\f x}^{(i)}}} \left( \sigma ^{j_1+j_2}\frac{(x_1-x^{(i)}_1)^{j_1}}{j_1 !}\frac{(x_2-x^{(i)}_2)^{j_2}}{j_2 !} \right),
\end{equation}
and $	\mathbb{I}_{\f x^{(i)}} = \left\{ \f j = (j_1,j_2): \; 0 \leq j_1,j_2 \mbox{ and } j_1+j_2\leq 2 \right\}$.
\end{definition}

\begin{remark}
  The factor $\sigma ^{j_1+j_2}$ in \eqref{eq:basisfctX-defn} is
  introduced to guarantee a uniform scaling, in $L^\infty$, of the
  basis functions. Indeed the vertex basis functions fulfill
\begin{displaymath}
	\partial_{x_1}^{m_1}\partial_{x_2}^{m_2} \left(\basisfctX{i}{\f j}\right)(\f x^{(i)}) = \sigma ^{j_1+j_2}\delta_{j_1}^{m_1}\;\delta_{j_2}^{m_2}
\end{displaymath}
for $0 \leq m_1\leq 2$, $0\leq m_2\leq 2$ and $m_1+m_2\leq 2$,  where $\delta_{j}^{m}$ is the Kronecker delta. 
\end{remark}
%

The construction is organized in two steps: In the first part we identify in each patch  a set of functions $\basisfct{\ell}{\f j}$ and
$\basisfctSigma{\ell}{\f j}$ with  specific interpolation properties at
the vertex ${\f x}^{(i)}$; then  we
combine   the functions above to define  global $C^1$ isogeometric
functions that allow  interpolation up to second order derivatives at the
vertex ${\f x}^{(i)}$. 
\begin{definition}
  Let $k$ be an even index. We define the space $\lW{k}$ of dimension $4$ as 
\begin{equation*}
\lW{k} = \rm{span} \left\{ \basisfct{\ii_k}{\f j}:
\f j \in \{0, 1\}^2 \right\},
\end{equation*}
where $\basisfct{\ii_k}{\f j}$ are given as in \eqref{eq:W-Omega-i}.
\end{definition}
\begin{lemma}\label{lem:vertex-patch-space}
For even  $k$, there exists a unique projector 
\begin{equation*}
\begin{array}{lll}
  \Pi_{\lW{k}} : &C^2\left( {\f x}^{(i)} \right) &\rightarrow \lW{k},
\end{array}
\end{equation*}
such that 
\begin{equation}\label{eq:intrpolation-prop-for-corner-subspace}
  \Du^{m_1} \Dv^{m_2} \, \left((\Pi_{\lW{k}}\varphi)\circ \f F^{(\ii_{k})}\right)(0,0) 
  = \Du^{m_1} \Dv^{m_2} \, \left(\varphi \circ \f F^{(\ii_{k})}\right)(0,0)
\end{equation}
for $0\leq m_1\leq 1$, $0\leq m_2\leq 1$.
\end{lemma}
\begin{proof}
The existence of such an operator follows from the classical interpolation properties of the standard tensor-product B-spline basis. 
\end{proof}
\begin{definition} 
  Let $k$ be an odd index. We define the space $\lW{k}$ of dimension $5$ as
\begin{equation*}
  \lW{k} = \rm{span} \left\{ \basisfctSigma{\ii_k}{\f j}:
  0\leq j_1, 0\leq j_2 \leq 1, j_1+j_2 \leq 2  \right\},
\end{equation*}
where the functions $\basisfctSigma{\ii_k}{\f j}$ are given as in
\eqref{eq:bj-edge-trace} {and \eqref{eq:bj-edge-derivative}} if
$\Sigma^{(\ii_k)}$ is an interface, or in \eqref{eq:boundary-bj-edge-trace} {and \eqref{eq:boundary-bj-edge-derivative}} if
$\Sigma^{(\ii_k)}$ is a boundary edge.
\end{definition}
\begin{lemma}\label{lem:vertex-edge-space}
Let $k$ be an odd index. 
There exists a unique projector  
\begin{equation*}
\begin{array}{lll}
  \Pi_{\lW{k}} : &C^2\left( {\f x}^{(i)}\right) &\rightarrow \lW{k},
\end{array}
\end{equation*}
such that if $\ii_{k-1} \in \indexOmega$, for $0\leq m_1\leq 1$, $0\leq
m_2\leq 2$ and $m_1+m_2\leq 2$ it holds
\begin{equation}\label{eq:vertex_interpolation_k-1}
  \Du^{m_1} \Dv^{m_2} \, \left((\Pi_{\lW{k}}\varphi)\circ \f F^{(\ii_{k-1})}\right)(0,0) 
  = \Du^{m_1} \Dv^{m_2} \, \left(\varphi \circ \f F^{(\ii_{k-1})}\right)(0,0);
\end{equation}
if $i_{k+1} \in \indexOmega$, for $0\leq m_1\leq 2$, $0\leq m_2\leq 1$
and $m_1+m_2\leq 2$ it holds
\begin{equation}\label{eq:vertex_interpolation_k+1}
  \Du^{m_1} \Dv^{m_2} \, \left((\Pi_{\lW{k}}\varphi)\circ \f F^{(\ii_{k+1})}\right)(0,0) 
  = \Du^{m_1} \Dv^{m_2} \, \left(\varphi \circ \f F^{(\ii_{k+1})}\right)(0,0).
\end{equation}
\end{lemma}
A proof of this lemma can be found in \ref{appendix:proofs}.

\begin{definition}[Vertex projector]\label{def:vertex-projector}
We define
\begin{equation}
  \label{eq:vertex-interpolant}
  \Pi_{\W_{{\f x}^{(i)}}} = \sum_{k=1}^{2\nu} (-1)^{k+1} \Pi_{\lW{k}} = \Pi_{\lW{1}}-  \Pi_{\lW{2}}
  \pm \ldots + \Pi_{\lW{2\nu-1}} -\Pi_{\lW{2\nu}}.
\end{equation}
With an abuse of notation, we assume   $\Pi_{\W_{{\f x}^{(i)}}}$
returns functions defined on the whole domain $\overline \Omega$,
extending to zero outside the support. 
\end{definition}
\begin{remark}
Since the projectors $\Pi_{\lW{k}}$ onto  $\lW{k}$ are defined by Hermite interpolation at the vertex ${\f x}^{(i)}$, the 
projector $\Pi_{\W_{{\f x}^{(i)}}}$ inherits its properties (such as the support size) from the basis functions 
$\basisfct{\ii_k}{\f j}$, for $k$ even, and $\basisfctSigma{\ii_k}{\f j}$, for $k$ odd.
\end{remark}
The projector $\Pi_{\W_{{\f x}^{(i)}}}$ in Definition \ref{def:vertex-projector} will satisfy the properties required in Lemma \ref{lem:full-vertex-proj-physical-intrpolation}. To complete the proof, we need one more preliminary result.
\begin{lemma}\label{lem:full-vertex-proj-parametric-intrpolation}
Let $k$ be any even integer index and $\varphi \in C^2({\f x}^{(i)})$. 
Then 
\begin{equation}\label{eq:vertex_interpolation_sum_parametric}
  \Du^{m_1} \Dv^{m_2} \, \left(( \Pi_{\W_{{\f x}^{(i)}}} \varphi)\circ \f F^{(\ii_{k})}\right)(0,0) 
  = \Du^{m_1} \Dv^{m_2} \, \left(\varphi \circ \f F^{(\ii_{k})}\right)(0,0)
\end{equation}
for $ m_1, m_2\geq 0 $ and $m_1+m_2\leq 2$. Moreover, if $\varphi\equiv 0$ then  $\Pi_{\W_{{\f x}^{(i)}}} \varphi\equiv 0$.
\end{lemma}
This lemma states that the projector interpolates up to second order derivatives, when mapped into the parameter domain. 
A proof of this lemma can be found in \ref{appendix:proofs}.

\begin{proof}[Proof of Lemma \ref{lem:full-vertex-proj-physical-intrpolation}]
Given $\varphi \in  C^2({\f x}^{(i)}) $, since $\Pi_{\W_{{\f
        x}^{(i)}}} \varphi$ is supported in the region
  $\overline\Omega^{(\ii_2)} \cup \ldots \cup \overline
  \Omega^{(\ii_{2\nu})}$, we need to consider the interfaces therein, 
 that are $ \overline\Sigma^{(\ii_{1} )}, \ldots ,  \overline\Sigma^{(\ii_{2\nu-1} )}$ for an
 interior vertex or  $ \overline\Sigma^{(\ii_{3} )}, \ldots ,
 \overline\Sigma^{(\ii_{2\nu-3} )}$ for a boundary vertex.
Then for an even index $k$, consider 
 a generic  $ \overline\Sigma^{(\ii_{k+1} )}$, with adjacent patches
 $\Omega^{(\ii_{k})}$ and $\Omega^{(\ii_{k+2})}$. We have 
 \begin{equation}
   \label{eq:vertex-projector-around-an-edge}
   \begin{aligned}
  \Pi_{\W_{{\f x}^{(i)}}} \varphi & = \Pi_{\lW{{k+1}}}
  \varphi -  \Pi_{\lW{{k}}} \varphi+ \Pi_{\lW{{k-1}}}
  \varphi,& \quad &\text{ on the patch }  \overline \Omega^{(\ii_{k})}, \\  \Pi_{\W_{{\f x}^{(i)}}} \varphi & = \Pi_{\lW{{k+3}}}
  \varphi -  \Pi_{\lW{{k+2}}} \varphi+ \Pi_{\lW{{k+1}}}
  \varphi, &\quad &\text{ on the patch }  \overline \Omega^{(\ii_{k+2})} .
   \end{aligned}
 \end{equation}

The term $-  \Pi_{\lW{{k}}} \varphi+ \Pi_{\lW{{k-1}}}\varphi$ has
vanishing trace and gradient at $\overline\Sigma^{(\ii_{k+1} )}$. Indeed,
adopting again the abbreviated notation as in the proof of Lemma \ref{lem:full-vertex-proj-parametric-intrpolation}, we have
\begin{equation}\label{eq:difference}
  (\Pi_{\lW{{k-1}}}\varphi)\circ \f F^{(\ii_{k})} - ( \Pi_{\lW{{k}}} \varphi)\circ \f F^{(\ii_{k})} = f^{(\ii_{k-1}, \ii_{k})}_h
      -f^{(\ii_{k}, \ii_{k})}_h 
\end{equation}
and, by plugging \eqref{eq: f^{(i_{k},i_{k})}} and  \eqref{eq:
  f^{(i_{k-1},i_{k})}} into the right hand side of \eqref{eq:difference}, it can be seen directly that the function and  its gradient  vanish at $(0, \xi)$, for all
$\xi \in [0,1]$. 
Similarly, $\Pi_{\lW{{k+3}}}   \varphi -  \Pi_{\lW{{k+2}}} \varphi$ has
vanishing trace and gradient at $\overline\Sigma^{(\ii_{k+1}  )}$. The only
remaining term in \eqref{eq:vertex-projector-around-an-edge} is $\Pi_{\lW{{k+1}}}
  \varphi$, which  is $C^1(\overline\Sigma^{(\ii_{k+1} )})$ by  Lemma~\ref{lemma:edge-fun-1}. 

The interpolation property  \eqref{eq:vertex_interpolation_sum}
follows  from
\eqref{eq:vertex_interpolation_sum_parametric}. Then, the dimension of
the image of $\Pi_{\W_{{\f   x}^{(i)}}} $ follows from Lemma \ref{lem:full-vertex-proj-parametric-intrpolation}.
\end{proof}

A dual vertex  basis is straightforwardly  derived by interpolation of derivatives up to second order.
\begin{definition}\label{def:dual-basis-vertex}
 The set $\{\basisfctXdual^{(i)}_{\f j}\}_{\f j \in \mathbb{I}_{{\f
       x}^{(i)}}}$, with  $$  \basisfctXdual^{(i)}_{\f j}(\varphi) = \frac{
 (\partial_{x_1}^{j_1}\partial_{x_2}^{j_2}\; \varphi )({\f
   x}^{(i)})}{ \sigma ^{j_1+j_2}}, $$  is  a dual basis for 
$\{\basisfctX{i}{\f j}\}_{\f j \in \mathbb{I}_{{\f x}^{(i)}}}$. 
\end{definition}
We have by definition $\Pi_{\W_{{\f x}^{(i)}}}  (\varphi) =
\sum_{\f j\in  {\mathbb{I}_{{\f x}^{(i)}}}} \basisfctXdual^{(i)}_{\f j}(\varphi)
\basisfctX{i}{\f j} $.

\subsection{Basis, dual basis and projector for the Argyris isogeometric space $\W$}

To summarize the results of this section, the functions 
\begin{itemize}
 \item $\{\basisfct{i}{\f j}\}_{\f j\in
  {\mathbb{I}^{\circ}_{\Omega^{(i)}}}}$, for $i\in\indexOmega$ (as in Definition \ref{lem:patch-space}),
 \item $\{\basisfctSigma{i}{\f j}\}_{\f j \in {\mathbb{I}^{\circ}_{\Sigma^{(i)}}}}$, for $i\in\indexSigma$ (as in Definitions \ref{def:edge-space} and \ref{def:boundary-edge-space}), and 
 \item $\{\basisfctX{i}{\f j}\}_{\f j \in \mathbb{I}_{{\f x}^{(i)}}}$, for $i\in\indexX$ (as in Definition \ref{def:vertex-basis}), 
\end{itemize}
form a basis for the space $\W$. 
The global projector 
\begin{displaymath}
\Pi_{\W} : C^2(\Omega) \rightarrow \W
\end{displaymath}
is defined via
\begin{displaymath}
  \Pi_{\W}(\varphi) = \sum_{i\in\indexOmega} \Pi_{\W^{\circ}_{\Omega^{(i)}}}(\varphi) + \sum_{i\in\indexSigma} \Pi_{\W^{\circ}_{\Sigma^{(i)}}}(\varphi) + \sum_{i\in\indexX} \Pi_{\W_{{\f x}^{(i)}}}(\varphi),
\end{displaymath}
where 
\begin{displaymath}
  \Pi_{\W^{\circ}_{\Omega^{(i)}}} (\varphi )= \sum_{\f j\in
  {\mathbb{I}^{\circ}_{\Omega^{(i)}}}} \Lambda_{\f j}(\varphi ) \basisfct{i}{\f j}
\end{displaymath}
as in Definition \ref{def:patch-dual-basis},
\begin{displaymath}
  \Pi_{\W^{\circ}_{\Sigma^{(i)}}} ( \varphi) = \sum_{\f j\in {\mathbb{I}^{\circ}_{\Sigma^{(i)}}}} \overline{\Lambda}^{(i)}_{\f j}(\varphi) \basisfctSigma{i}{\f j}
\end{displaymath}
as in Definition \ref{def:edge-space-dual-proj}, and 
\begin{displaymath}
\Pi_{\W_{{\f x}^{(i)}}}  (\varphi) =
\sum_{\f j\in  {\mathbb{I}_{{\f x}^{(i)}}}} \basisfctXdual^{(i)}_{\f j}(\varphi)
\basisfctX{i}{\f j} 
\end{displaymath}
as in Lemma \ref{lem:full-vertex-proj-physical-intrpolation} and Definition \ref{def:dual-basis-vertex}.

The following result follows directly from the definition of the space $\W$.
\begin{proposition}\label{prop:smoothness-of-W}
We have for $\varphi_h \in \W$, that $\varphi_h \in C^1(\Omega)$ and
$\varphi_h \in C^2({\f x}^{(i)})$ for all $i\in\indexX$.
\end{proposition}

\section{Numerical examples}\label{sec:tests}

We perform $L^2$-approximation over two AS-$G^1$ multi-patch parametrizations to numerically show that the Argyris isogeometric space $\W$ maintains the polynomial reproduction 
properties of the entire space~$\mathcal{V}^1$ for the traces and normal derivatives along the interfaces, and that the space~$\W$, being a subspace of $\mathcal{V}^1$, produces relative $L^2$ errors of the same magnitude 
as the entire space $\mathcal{V}^1$. 

For this purpose, we consider the two AS-$G^1$ multi-patch parametrizations visualized in Fig.~\ref{fig:AS_geometries} (left). Both AS-$G^1$ geometries consist of single spline 
patches~$\f F^{(i)} \in \Spr \times \Spr$ with $\f p=(3,3)$, $\f r=(1,1)$ and $h=\frac{1}{2}$, and are generated by using the method presented in~\cite{KaSaTa17b}. The construction 
of the AS-$G^1$ five patch parametrization was already demonstrated in~\cite[Example 1]{KaSaTa17b}. The AS-$G^1$ three-patch parametrization can be obtained in an analogous manner. 

For both AS-$G^1$ geometries we generate a sequence of nested spaces $\W_h$ and $\mathcal{V}_h^1$ for $\f p=(3,3)$ and $\f r=(1,1)$ by selecting the mesh size $h$ as $\frac{1}{4}$, 
$\frac{1}{8}$, $\frac{1}{16}$ and $\frac{1}{32}$. While the bases of the Argyris spaces~$\W_h$ are simply constructed as described in Section~\ref{sec:W}, the bases of the entire 
spaces~$\mathcal{V}^1_h$ are obtained in the same way as in \cite[Section 4.2]{KaSaTa17b} by means of the concept of minimal determining sets (cf. \cite{LaSch07}). Note that in contrast 
to the basis functions of $\W_h$, the resulting basis functions of $\mathcal{V}^1_h$ are in general not locally supported and possess a support over at least one entire interface.

We use now the basis functions for the spaces~$\W_h$ and $\mathcal{V}^1_h$ to perform $L^2$-approximation over the two AS-$G^1$ multi-patch parametrizations. Consider one 
of the spaces~$\W_h$ or $\mathcal{V}^1_h$, and let $\phi_{\f j}$ be the corresponding basis functions. The goal is to approximate the function
\begin{equation} \label{eq:exact_solution}
 z: \Omega \rightarrow \RR , \quad z(\f x) = z(x_1,x_2) = 2 \cos(x_1) \sin (x_2),
\end{equation}
see Fig.~\ref{fig:AS_geometries} (middle), by the function 
\[
 u_h (\f{x}) = \sum_{\f j} c_{\f j} \phi_{\f j}(\f{x}) , \quad c_{\f j} \in \RR,
\]
via minimizing the term 
\[
\|u_h-z\|_{L^2}^{2} = \int_{\Omega} (u_{h}(\f{x})- z(\f{x}) )^{2} \mathrm{d}\f{x} 
\rightarrow  \min_{c_{\f j}}.
\]
The isogeometric formulation of this linear problem was discussed in detail  in \cite[Section 4.2]{KaSaTa17b}, and will be omitted here for the sake of brevity.

\begin{figure}
\centering\footnotesize 
\begin{tabular}{ccc}
AS-$G^1$ geometry & Exact solution & Relative $L^2$-error \\  
\includegraphics[width=5.0cm,clip]{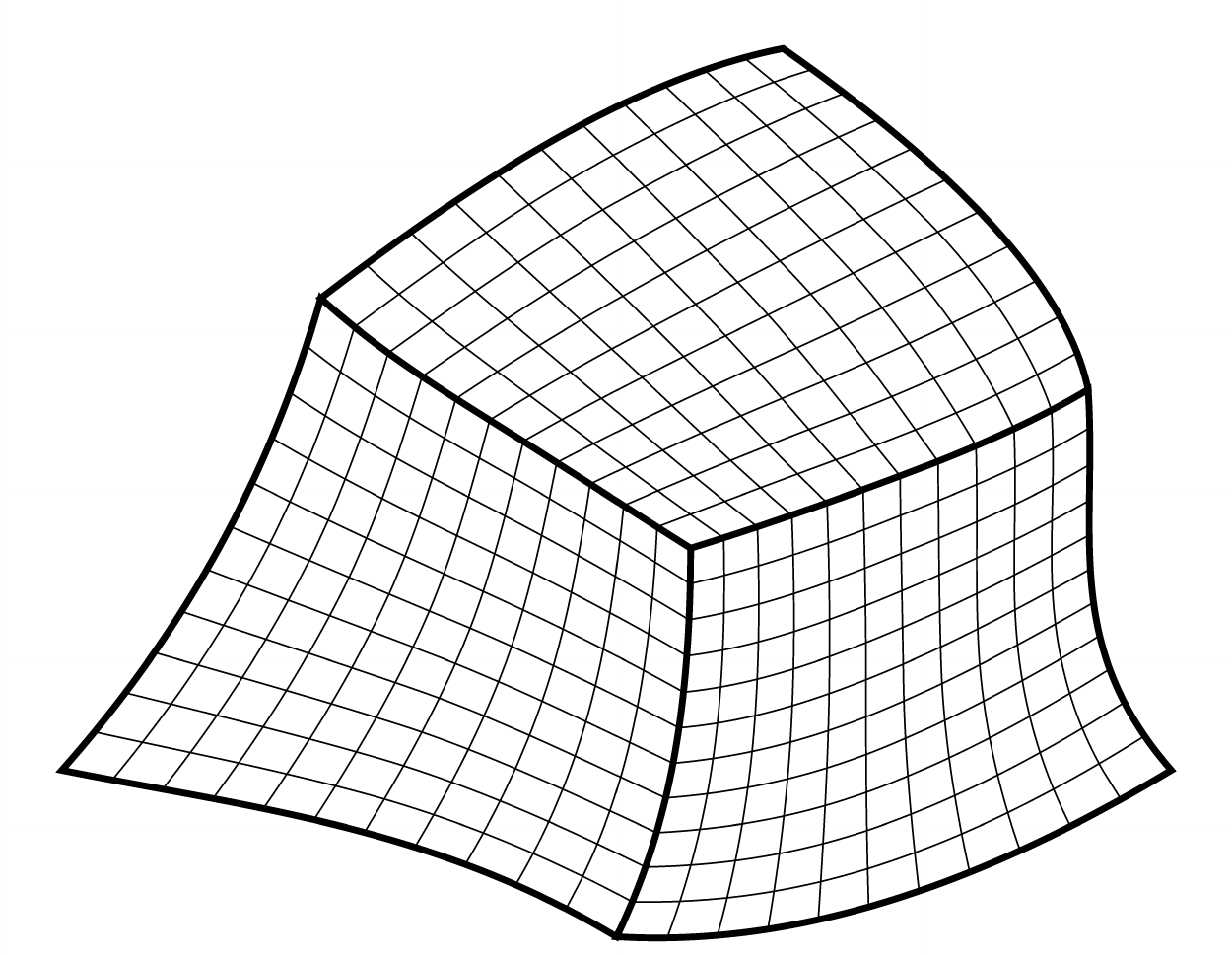} &
\includegraphics[width=4.7cm,clip]{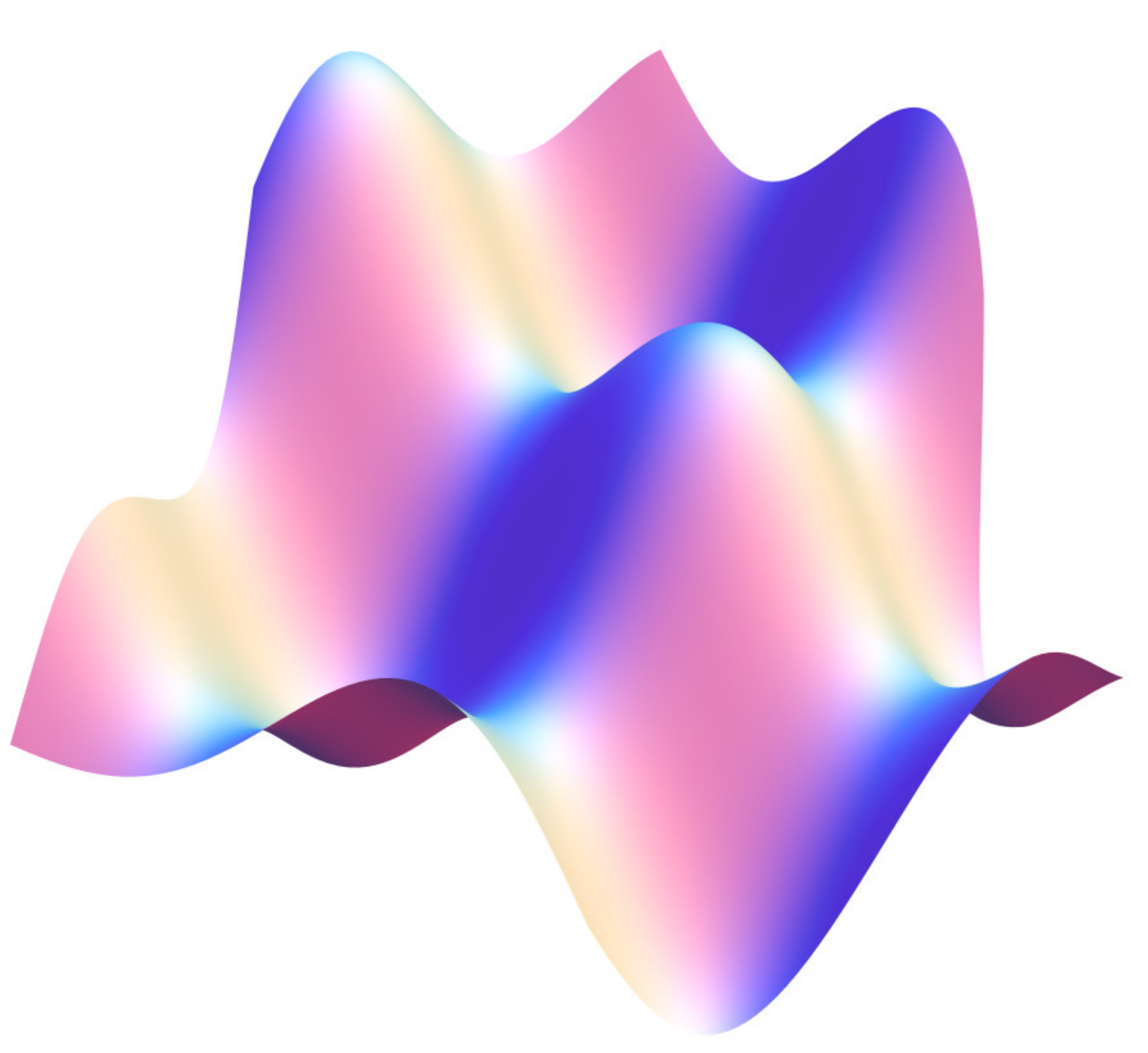} &
\includegraphics[width=5.5cm,clip]{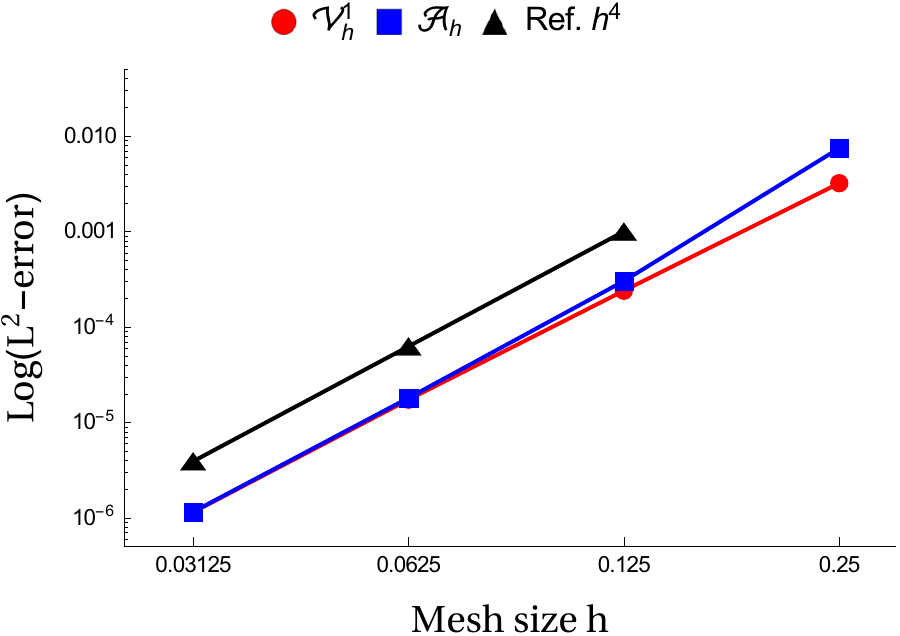} \\
\multicolumn{3}{c}{Example: AS-$G^1$ three-patch geometry} \\
\includegraphics[width=5.0cm,clip]{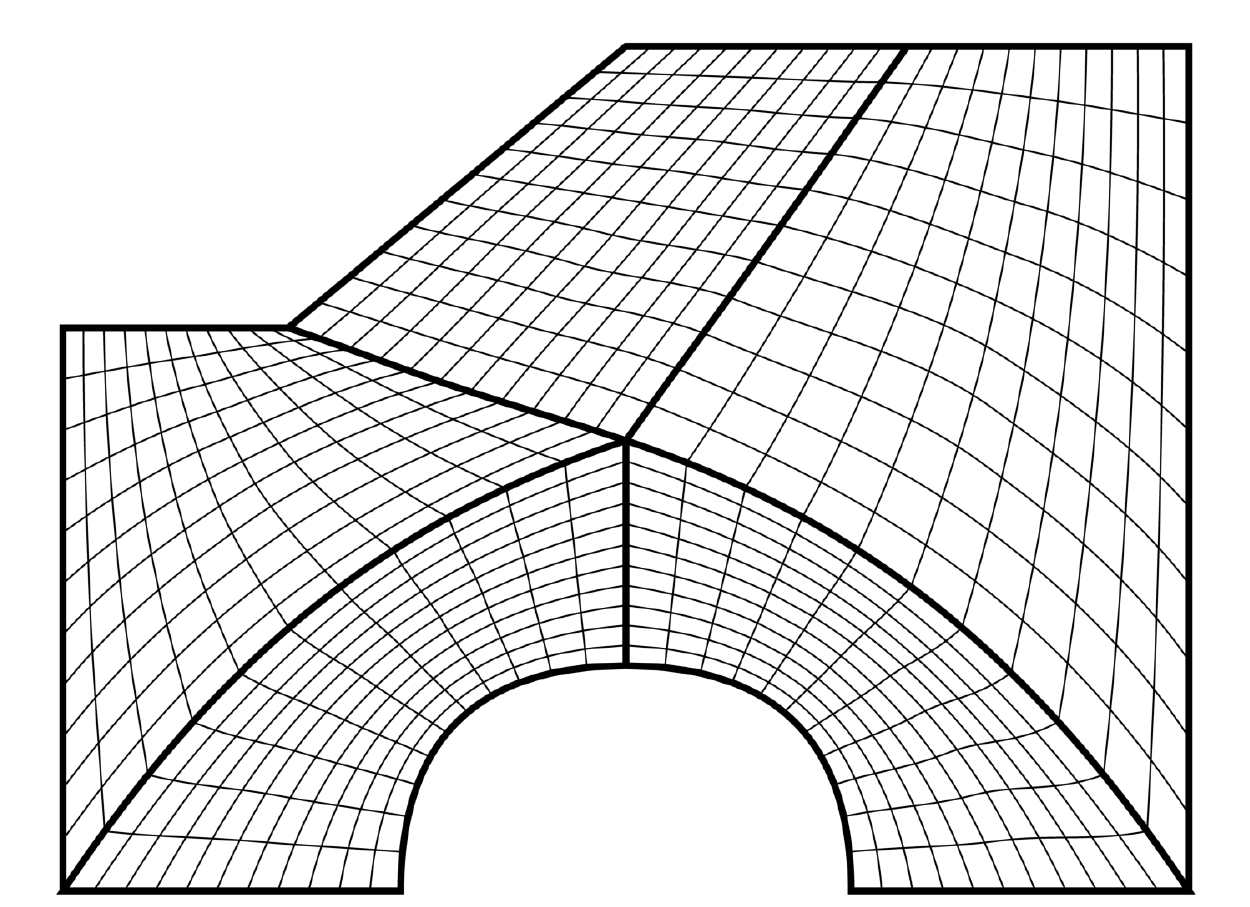} &
\includegraphics[width=4.7cm,clip]{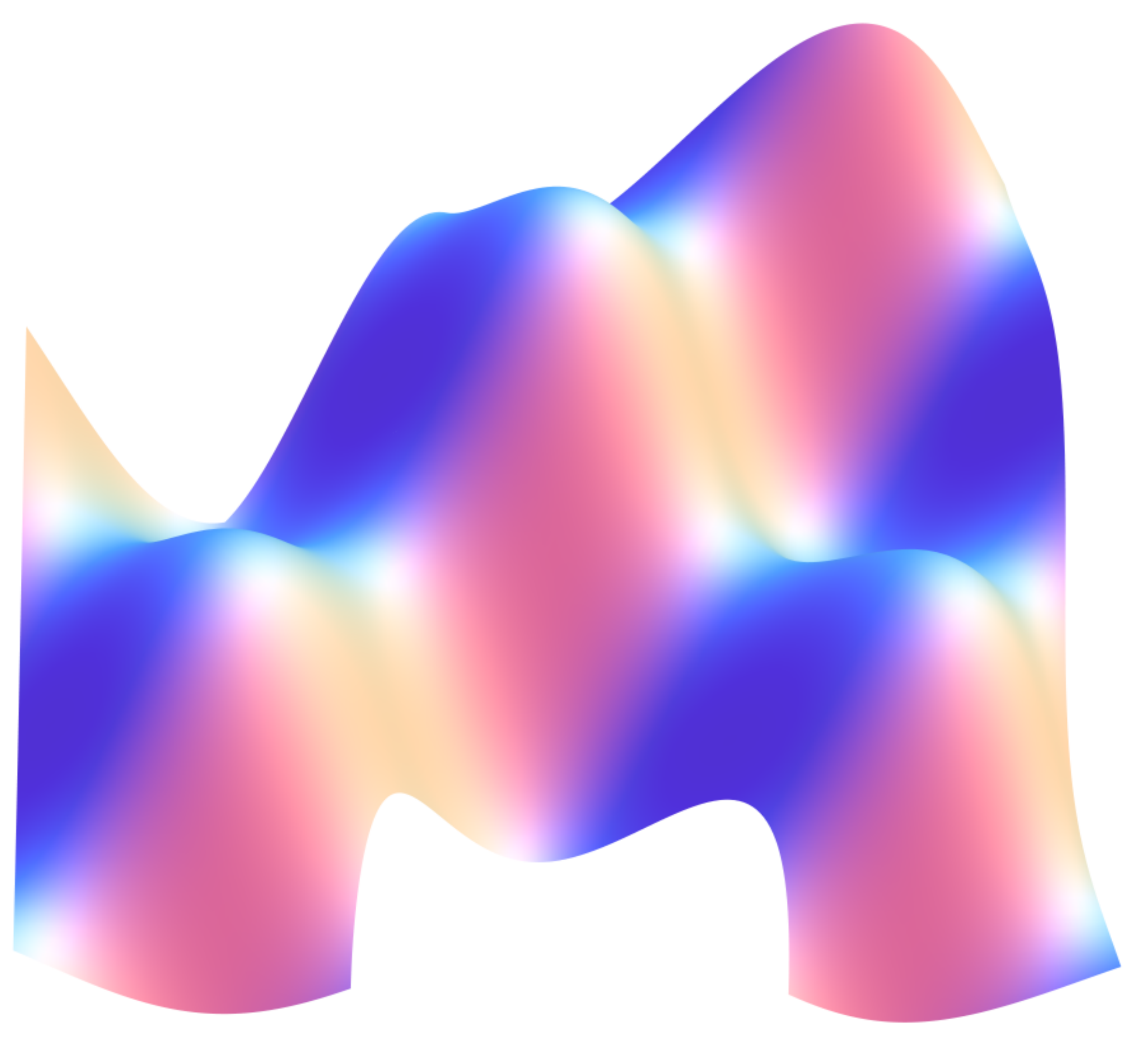} &
\includegraphics[width=5.5cm,clip]{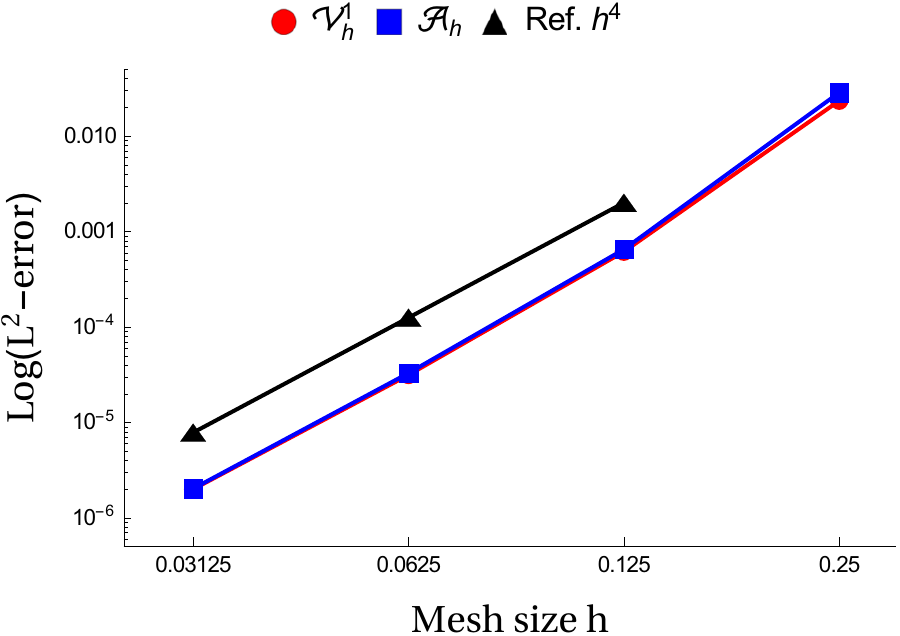} \\
\multicolumn{3}{c}{Example: AS-$G^1$ five-patch geometry}
\end{tabular}
\caption{$L^2$-projection over two AS-$G^1$ multi-patch parametrizations (left) by using the two different spaces~$\W_h$ and $\mathcal{V}^1_h$ (right) to approximate the 
function~\eqref{eq:exact_solution} (middle), cf. Table~\ref{tab:L2}.}
\label{fig:AS_geometries}
\end{figure}

Table~\ref{tab:L2} and Fig.~\ref{fig:AS_geometries}~(right) report the resulting relative $L^2$-errors and the estimated convergence rates for the two spaces~$\W_h$ and $\mathcal{V}^1_h$ for 
the different mesh sizes~$h$. The numerical results indicate for both spaces convergence rates of optimal order~$\mathcal{O}(h^4)$ in the $L^2$-norm and show that resulting relative 
$L^2$-errors are of the same magnitude for the two spaces. 

\begin{table}
  \centering\scriptsize \advance\tabcolsep by -4pt 
  \begin{tabular}{|c||c|c|c||c|c|c|} \hline
  & \multicolumn{3}{|c||}{Subspace $\W$} & \multicolumn{3}{|c|}{Entire space~$\mathcal{V}^1$} \\ \hline \hline
  & \multicolumn{6}{|c|}{AS-$G^1$ three-patch geometry (a)} \\ \hline
  $h$  & $\s \dim \W$ & $\s \frac{||u_{h}-z||_{L^{2}}}{||z||_{L^{2}}}$ & e.c.r. $\s  ||\cdot||_{L^{2}}$ 
   & $\s \dim \mathcal{V}^{1}$ & $\s \frac{||u_{h}-z||_{L^{2}}}{||z||_{L^{2}}}$ & e.c.r. $\s  ||\cdot||_{L^{2}}$ \\ \hline \hline
  1/4  & 177   & 7.46e-03 & -    & 222   & 3.21e-03 & -       \\ \hline 
  1/8  & 729   & 3.03e-04 & 4.62 & 822   & 2.4e-04  & 3.74    \\ \hline
  1/16 & 2985  & 1.8e-05  & 4.07 & 3174  & 1.73e-05 & 3.79    \\ \hline
  1/32 & 12105 & 1.15e-06 & 3.97 & 12486 & 1.14e-06 & 3.92    \\ \hline \hline
   & \multicolumn{6}{|c|}{AS-$G^1$ five-patch geometry (b)} \\ \hline
  $h$  & $\s \dim \W$ & $\s \frac{||u_{h}-z||_{L^{2}}}{||z||_{L^{2}}}$ & e.c.r. $\s  ||\cdot||_{L^{2}}$ 
  & $\s \dim \mathcal{V}^{1}$ & $\s \frac{||u_{h}-z||_{L^{2}}}{||z||_{L^{2}}}$ & e.c.r. $\s  ||\cdot||_{L^{2}}$ \\ \hline \hline
  1/4  & 291   & 2.86e-02 & -    & 372   & 2.35e-02 & -      \\ \hline 
  1/8  & 1211  & 6.5e-04  & 5.46 & 1376  & 6.14e-04 & 5.25   \\ \hline
  1/16 & 4971  & 3.28e-05 & 4.31 & 5304  & 3.14e-05 & 4.29   \\ \hline
  1/32 & 20171 & 2.02e-06 & 4.02 & 20840 & 1.99e-06 & 3.98  \\ \hline
  \end{tabular}
  \caption{Resulting relative $L^2$-errors with estimated convergence rates of the diagonally scaled mass matrices by 
  performing $L^2$-approximation over two AS-$G^1$ multi-patch geometries using the two $C^1$ spaces $\W$ and $\mathcal{V}^1$, cf. Fig.~\ref{fig:AS_geometries}.}
  \label{tab:L2}
 \end{table}

\section{Conclusion} \label{sec:conclusion}

We presented for the class of AS-$G^1$ multi-patch parametrizations the construction of a basis and of an associated dual basis for the so-called Argyris isogeometric space $\W$, which generalizes the classical Argyris finite elements to multi-patch isogeometric spaces. It is 
a subspace of the entire $C^1$  isogeometric space~$\mathcal{V}^{1}$ maintaining the polynomial reproduction properties of~$\mathcal{V}^{1}$ for the traces and normal 
derivatives along the interfaces. This property of the subspace~$\W$ was shown numerically by performing $L^2$-approximation over different AS-$G^1$ multi-patch parametrizations.
The use of the Argyris space~$\W$ instead of the space~$\mathcal{V}^{1}$ is advantageous since the subspace~$\W$ has a simpler structure and allows a uniform and simple 
construction of the basis functions independent of the AS-$G^1$ domain parametrization. The construction of the basis (and of its dual basis) is based on the decomposition of the 
space~$\W$ into the direct sum of three subspaces called the patch-interior, the edge and the vertex function space. The resulting basis and the dual basis have a simple 
form, since the single functions are locally supported and are explicitly given by closed form representations. 

This paper presents the foundation for further studies of  $C^1$ isogeometric spaces over AS-$G^1$ multi-patch
parametrizations, by providing a basis and corresponding projectors.
A first planned topic for future research is the theoretical investigation of the properties of the 
space~$\W$, such as approximation error and stability estimates for $h$-refined meshes, which can be built upon a suitable dual basis. Moreover, one may also construct a basis forming a partition of unity, following the ideas presented in \cite{dierckx1997calculating}, based on local triangular B\'ezier surfaces at the vertices.
We are also planning to extend the construction to surface domains and to use our approach to perform Kirchhoff-Love shell analysis for 
different linear and non-linear model configurations. Another challenging task will be the extension to volumetric domains. So far, no 
generalization of AS-$G^1$ parametrizations to volumetric domains is known.

\section*{Acknowledgments}
G. Sangalli is member of the  Gruppo Nazionale Calcolo
Scientifico-Istituto Nazionale di Alta Matematica (GNCS-INDAM), and
was partially supported by the European Research Council 
through the FP7 Ideas Consolidator Grant \emph{HIGEOM} n.616563. This support is
gratefully acknowledged. 

\appendix
\section{Proof of Lemma \ref{lem:vertex-edge-space} and \ref{lem:full-vertex-proj-parametric-intrpolation}}\label{appendix:proofs}

\begin{proof}[Proof of Lemma \ref{lem:vertex-edge-space}]
Let $\phi_{m_1,m_2} = \partial_{x_1}^{m_1} \partial_{x_2}^{m_2}\varphi({\f x}^{(i)})$. 
Assume $\ii_{k-1} , \ii_{k+1}\in \indexOmega$ (this is not true in general
for boundary vertices, if so the proof below simplifies in a trivial
way). On each patch $\Omega^{(\ii_\ell)}$, $\ell = k-1$ or $\ell=k+1$, we can
write the pull-back of $\varphi$ and of the projection
$\Pi_{\lW{k}}\varphi$ as 
\begin{equation*}
	f^{(\ii_\ell)}(\xi_1,\xi_2) = \varphi \circ \f F^{(\ii_\ell)}(\xi_1,\xi_2)
\end{equation*}
and
\begin{equation*}
	f^{(\ii_k,\ii_\ell)}_h(\xi_1,\xi_2) = (\Pi_{\lW{k}}\varphi) \circ \f F^{(\ii_\ell)}(\xi_1,\xi_2).
\end{equation*}
We have by definition, using the abbreviations $\nabla \phi = (\phi_{1,0},\phi_{0,1})$ and 
\begin{equation*}
H\phi = \left(
\begin{array}{ll}
  \phi_{2,0} & \phi_{1,1} \\
  \phi_{1,1} & \phi_{0,2} 
\end{array}\right),
\end{equation*}
and using the chain rule of differentiation, that
\begin{equation}\label{eq:fphiF}
\begin{array}{lll}
  f^{(\ii_\ell)}(0,0) &=& \phi_{0,0}, \\
  \Du f^{(\ii_\ell)}(0,0)  &=& \nabla \phi \Du \f F^{(\ii_\ell)}(0,0),\\
  \Dv f^{(\ii_\ell)}(0,0)  &=& \nabla \phi \Dv \f F^{(\ii_\ell)}(0,0),\\
  \Du^2 f^{(\ii_\ell)}(0,0) 
  &=& (\Du \f F^{(\ii_\ell)} (0,0))^T 
  \; H\phi \; \Du \f F^{(\ii_\ell)} (0,0)
  + \nabla\phi \; \Du \Du \f F^{(\ii_\ell)}(0,0),\\
  \Du\Dv f^{(\ii_\ell)}(0,0) 
  &=& (\Du \f F^{(\ii_\ell)}(0,0))^T 
  \; H\phi \; \Dv \f F^{(\ii_\ell)} (0,0)
  + \nabla\phi \; \Du \Dv \f F^{(\ii_\ell)}(0,0),\\
  \Dv^2 f^{(\ii_\ell)}(0,0) 
  &=& (\Dv \f F^{(\ii_\ell)}(0,0))^T 
  \; H\phi \; \Dv \f F^{(\ii_\ell)} (0,0) 
  + \nabla\phi \; \Dv \Dv \f F^{(\ii_\ell)}(0,0).
\end{array}
\end{equation}
Consider the basis transformations from $\{b_0^+,b_1^+,b_2^+\}$ to $\{\ibasis_0^+,\ibasis_1^+,\ibasis_2^+\}$ and 
from $\{b_0^-,b_1^-\}$ to $\{\ibasis_0^-,\ibasis_1^-\}$,
with
\begin{equation*}
  \partial^j_{\xi}\ibasis_i^+(0)=\delta_i^j \quad\mbox{for }j=0,\ldots,2,\quad\mbox{and}\quad   \partial^j_{\xi} \ibasis_i^-(0)=\delta_i^j \quad\mbox{for }j=0,1,
\end{equation*}
where $\delta_i^j$ is the Kronecker delta. 
Then, recalling \eqref{eq:bj-edge-trace} and \eqref{eq:bj-edge-derivative}, we can rewrite the functions $ f^{(\ii_k,\ii_{k-1})}_h$ and $
f^{(\ii_k,\ii_{k+1})}_h$ in terms of the new bases:
\begin{equation}\label{eq:change-basis-k-1}
\begin{array}{lll}
  f^{(\ii_k,\ii_{k-1})}_h(\xi_1,\xi_2) &=& \sum_{j=0}^2 d_{0,j}\, \left(\ibasis_j^+(\xi_2)\ibasis_0(\xi_1)-\beta^{(\ii_k,\ii_{k-1})}(\xi_2) (\ibasis^+_{j})'(\xi_2) \ibasis_1(\xi_1) \right)\\
  &+& \sum_{j=0}^1 d_{1,j}\, \alpha^{(\ii_k,\ii_{k-1})}(\xi_2)
      \ibasis^-_{j}(\xi_2) b_1(\xi_1),
\end{array}
\end{equation}
and 
\begin{equation}\label{eq:change-basis-k+1}
\begin{array}{lll}
  f^{(\ii_k,\ii_{k+1})}_h(\xi_1,\xi_2) &=& \sum_{j=0}^2 d_{0,j}\, \left(\ibasis_j^+(\xi_1)\ibasis_0(\xi_2)-\beta^{(\ii_k,\ii_{k+1})}(\xi_1) (\ibasis^+_{{j}})'(\xi_1) \ibasis_1(\xi_2) \right)\\
  &-& \sum_{j=0}^1 d_{1,j}\, \alpha^{(\ii_k,\ii_{k+1})}(\xi_1)
      \ibasis^-_{{j}}(\xi_1) b_1(\xi_2).
\end{array}
\end{equation}
Considering \eqref{eq:change-basis-k-1}, we then have
\begin{equation}\label{eq:derivatives_for_f}
\begin{array}{lll}
  f^{(\ii_k,\ii_{k-1})}_h(0,0) &=& d_{0,0}, \\
  \Dv f^{(\ii_k,\ii_{k-1})}_h(0,0)  &=& d_{0,1}, \\
  \Dv^2 f^{(\ii_k,\ii_{k-1})}_h(0,0) &=& d_{0,2}, \\
  \Du f^{(\ii_k,\ii_{k-1})}_h(0,0)  &=&  \frac{p}{h} \alpha^{(\ii_k,\ii_{k-1})}(0)\, d_{1,0} - \beta^{(\ii_k,\ii_{k-1})}(0) \, d_{0,1}, \\
  \Du\Dv f^{(\ii_k,\ii_{k-1})}_h(0,0) &=&  \frac{p}{h} \alpha^{(\ii_k,\ii_{k-1})}(0)\, d_{1,1} - \beta^{(\ii_k,\ii_{k-1})}(0) \, d_{0,2} \\
  && \; +  \frac{p}{h}  (\alpha^{(\ii_k,\ii_{k-1})})'(0)\, d_{1,0} - (\beta^{(\ii_k,\ii_{k-1})})'(0) \, d_{0,1} ;
\end{array}
\end{equation}
using the abbreviated notation 
\begin{equation*}
  \f t^{(\ii_k)}(\xi) = \Dv \f F^{(\ii_{k-1})}(0,\xi) = \Du \f F^{(\ii_{k+1})}(\xi,0)
\end{equation*}
and 
\begin{equation*}
\begin{array}{lll}
  \fn^{(\ii_k)}(\xi) &=& \frac{1}{\alpha^{(\ii_k,\ii_{k-1})}(\xi)}\left( \Du \f F^{(\ii_{k-1})}(0,\xi) + \beta^{(\ii_k,\ii_{k-1})}(\xi) \, \Dv \f F^{(\ii_{k-1})}(0,\xi)\right) \\
  &=& -\frac{1}{\alpha^{(\ii_k,\ii_{k+1})}(\xi)}\left( \Dv \f F^{(\ii_{k+1})}(\xi,0) + \beta^{(\ii_k,\ii_{k+1})}(\xi) \, \Du \f F^{(\ii_{k+1})}(\xi,0)\right)
\end{array}
\end{equation*}
we can determine all $d_{i,j}$ from the interpolation conditions
\begin{displaymath}
\begin{array}{lll}
  f^{(\ii_k,\ii_{k-1})}_h(0,0) &=& f^{(\ii_{k-1})}(0,0) , \\
  \Dv f^{(\ii_k,\ii_{k-1})}_h(0,0)  &=&\Dv f^{(\ii_{k-1})}(0,0) , \\
  \Dv^2 f^{(\ii_k,\ii_{k-1})}_h(0,0) &=&  \Dv^2 f^{(\ii_{k-1})}(0,0) , \\
  \Du f^{(\ii_k,\ii_{k-1})}_h(0,0)  &=& \Du f^{(\ii_{k-1})}(0,0) , \\
  \Du\Dv f^{(\ii_k,\ii_{k-1})}_h(0,0) &=&   \Du\Dv f^{(\ii_{k-1})}(0,0) 
\end{array}
\end{displaymath}
and   both \eqref{eq:fphiF} and \eqref{eq:derivatives_for_f}:
\begin{equation}\label{eq:dphi}
\begin{array}{lll}
  d_{0,0} &=& \phi_{0,0}, \\
  d_{0,1} &=& \nabla \phi \; \f t^{(\ii_k)}(0), \\
  d_{0,2} &=& (\f t^{(\ii_k)}(0))^T 
  \; H\phi \; \f t^{(\ii_k)}(0)
  + \nabla\phi \; (\f t^{(\ii_k)})'(0), \\
   \frac{p}{h}  d_{1,0} &=& \nabla \phi \; \fn^{(\ii_k)}(0), \\
   \frac{p}{h}  d_{1,1} &=& (\f t^{(\ii_k)}(0))^T \; H\phi \; \fn^{(\ii_k)}(0) + \nabla\phi \;(\fn^{(\ii_k)})'(0). 
\end{array}
\end{equation}
Hence, there exists a projector $\Pi_{\lW{k}}$ satisfying
\eqref{eq:vertex_interpolation_k-1}. 
We can reason similarly on \eqref{eq:change-basis-k+1}, where we have
\begin{equation}\label{eq:fd}
\begin{array}{lll}
  f^{(\ii_k,\ii_{k+1})}_h(0,0) &=& d_{0,0}, \\
  \Du f^{(\ii_k,\ii_{k+1})}_h(0,0)  &=& d_{0,1},\\
  \Du\Du f^{(\ii_k,\ii_{k+1})}_h(0,0) &=& d_{0,2},\\
  \Dv f^{(\ii_k,\ii_{k+1})}_h(0,0)  &=& -  \frac{p}{h} \alpha^{(\ii_k,\ii_{k+1})}(0)\, d_{1,0} - \beta^{(\ii_k,\ii_{k+1})}(0) \, d_{0,1},\\
  \Du\Dv f^{(\ii_k,\ii_{k+1})}_h(0,0) &=& -  \frac{p}{h} \alpha^{(\ii_k,\ii_{k+1})}(0)\, d_{1,1} - \beta^{(\ii_k,\ii_{k+1})}(0) \, d_{0,2} \\
  && \; - \frac{p}{h}  (\alpha^{(\ii_k,\ii_{k+1})})'(0)\, d_{1,0} - (\beta^{(\ii_k,\ii_{k+1})})'(0) \, d_{0,1}.
\end{array}
\end{equation}
In fact, after inserting \eqref{eq:dphi} into \eqref{eq:fd} and simplifying, we obtain \eqref{eq:fphiF} for $\ell=k+1$. 
Hence, the projector $\Pi_{\lW{k}}$ also satisfies \eqref{eq:vertex_interpolation_k+1}, which concludes the proof.
\end{proof}

\begin{proof}[Proof of Lemma \ref{lem:full-vertex-proj-parametric-intrpolation}]
We use again the abbreviated notation $
f^{(\ii_k)} = \varphi \circ \f
F^{(\ii_k)} $ and $
f^{(\ii_\ell,\ii_k)}_h= (\Pi_{\lW{l}}\varphi) \circ \f
F^{(\ii_k)}$, where by definition $f^{(\ii_\ell,\ii_k)}_h = 0 $ for $\ell
\neq k-1, k, k+1$.

 Then we have
\begin{equation}
  \label{eq:vertex_function_on_a_patch}
  (\Pi_{\W_{{\f x}^{(i)}}} \varphi)\circ \f F^{(\ii_{k})} =f^{(\ii_{k-1},\ii_k)}_h -f^{(\ii_{k},\ii_k)}_h +f^{(\ii_{k+1},\ii_k)}_h .
\end{equation}
The active degrees-of-freedom of three  terms in
\eqref{eq:vertex_function_on_a_patch} with respect to the underlying
tensor-product spline space $\Spr $ are pictured in Figure~\ref{fig:vertex_basis_representation}.
\begin{figure}[ht]
  \centering
  \begin{picture}(120,120)
    \put(0,0){\includegraphics[width=.25\textwidth]{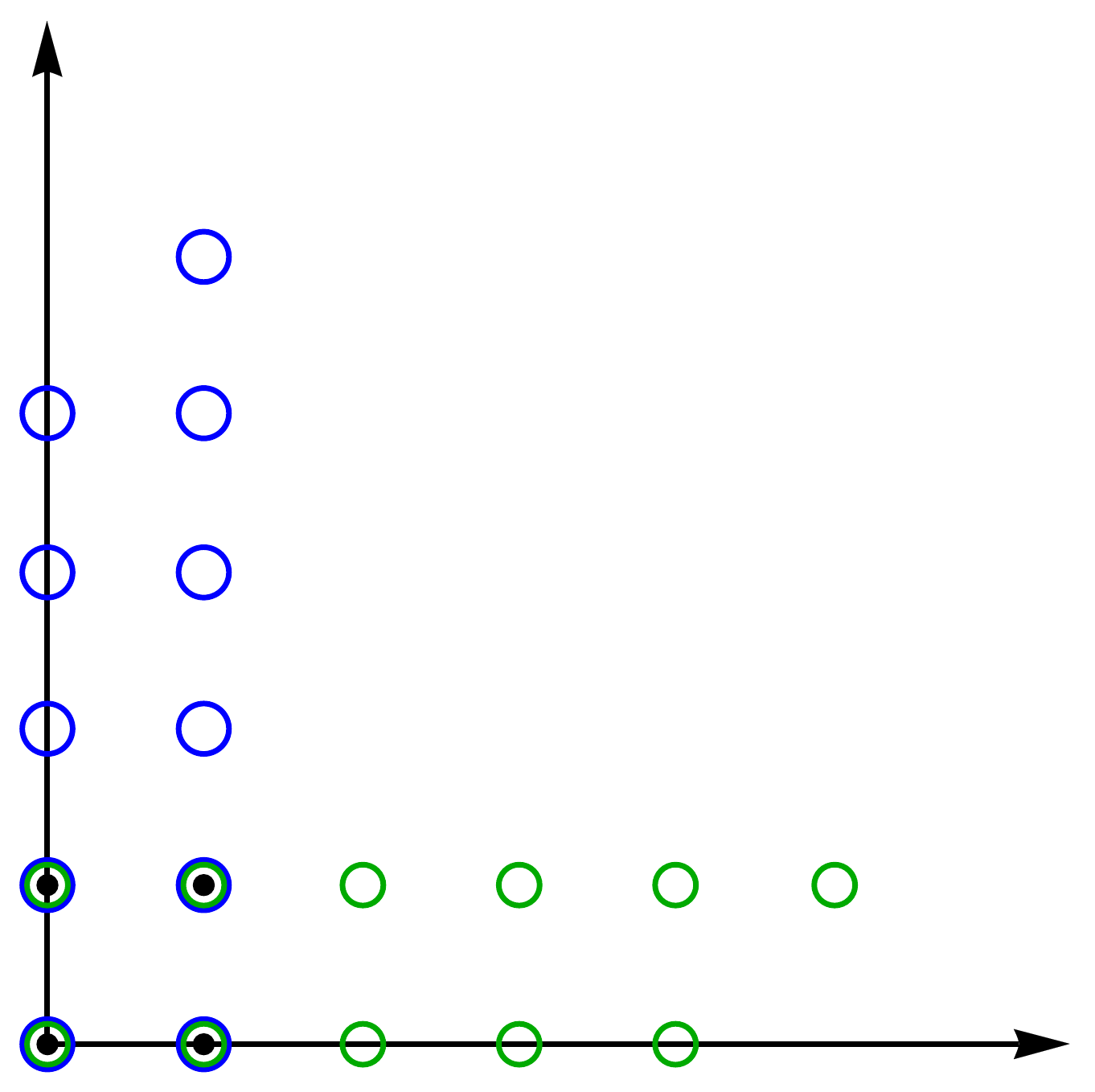}}
    \put(110,15){$\xi_1$}
    \put(15,110){$\xi_2$}
  \end{picture}
  \caption{Active degrees-of-freedom of the functions  in
\eqref{eq:vertex_function_on_a_patch} with respect to the 
tensor-product spline space $\Spr $: The degrees of freedom corresponding to $f^{(\ii_{k-1},\ii_k)}_h $ depicted in green, corresponding to 
$f^{(\ii_{k},\ii_k)}_h $ in black and corresponding to $f^{(\ii_{k+1},\ii_k)}_h $ in blue.}
  \label{fig:vertex_basis_representation}
\end{figure}

By definition and thanks  to the 
interpolation properties
\eqref{eq:intrpolation-prop-for-corner-subspace} and
\eqref{eq:vertex_interpolation_k-1}--\eqref{eq:vertex_interpolation_k+1}
we have:
\begin{equation*}
  \label{eq: f^{(i_{k+1},i_{k})}}
  \begin{aligned}
       f^{(\ii_{k+1},\ii_{k})}_h(0,0) &=  f^{(\ii_{k})}(0,0) , \\
  \Dv f^{(\ii_{k+1},\ii_{k})}_h(0,0)  &= \Dv f^{(\ii_{k})}(0,0) , \\
  \Dv^2 f^{(\ii_{k+1},\ii_{k})}_h(0,0) &=   \Dv^2 f^{(\ii_{k})}(0,0) , \\
  \Du f^{(\ii_{k+1},\ii_{k})}_h(0,0)  &=  \Du f^{(\ii_{k})}(0,0) , \\
  \Du\Dv f^{(\ii_{k+1},\ii_{k})}_h(0,0) &=    \Du\Dv f^{(\ii_{k})}(0,0), \\
 \Du^2 f^{(\ii_{k+1},\ii_{k})}_h(0,0) &=   0,
  \end{aligned}
\end{equation*}
and
\begin{equation*}
  \label{eq: f^{(i_{k},i_{k})}}
  \begin{aligned}
  f^{(\ii_k,\ii_{k})}_h(0,0) &=  f^{(\ii_{k})}(0,0) , \\
  \Du f^{(\ii_k,\ii_{k})}_h(0,0)  &=  \Du f^{(\ii_{k})}(0,0) , \\
  \Dv f^{(\ii_k,\ii_{k})}_h(0,0)  &= \Dv f^{(\ii_{k})}(0,0) , \\
  \Du\Dv f^{(\ii_k,\ii_{k})}_h(0,0) &=    \Du\Dv f^{(\ii_{k})}(0,0),\\
 \Du^2 f^{(\ii_k,\ii_{k})}_h(0,0) &=  0 , \\
  \Dv^2 f^{(\ii_k,\ii_{k})}_h(0,0) &=  0,
  \end{aligned}
\end{equation*}
and 
\begin{equation*}
  \label{eq: f^{(i_{k-1},i_{k})}}
  \begin{aligned}
   f^{(\ii_{k-1},\ii_{k})}_h(0,0) &=  f^{(\ii_{k})}(0,0) , \\
  \Du f^{(\ii_{k-1},\ii_{k})}_h(0,0)  &=  \Du f^{(\ii_{k})}(0,0) , \\
  \Du^2 f^{(\ii_{k-1},\ii_{k})}_h(0,0) &=   \Du^2 f^{(\ii_{k})}(0,0) , \\
  \Dv f^{(\ii_{k-1},\ii_{k})}_h(0,0)  &= \Dv f^{(\ii_{k})}(0,0) , \\
  \Du\Dv f^{(\ii_{k-1},\ii_{k})}_h(0,0) &=    \Du\Dv f^{(\ii_{k})}(0,0),
  \\  \Dv^2 f^{(\ii_{k-1},\ii_{k})}_h(0,0) &=  0.
  \end{aligned}
\end{equation*}

Using the above  relations into \eqref{eq:vertex_function_on_a_patch}
we get~\eqref{eq:vertex_interpolation_sum_parametric}.
  
Finally, if  $ f^{(\ii_{k})}$ vanishes, then all the derivatives above
are null and, by definition, $  f^{(\ii_{k-1},\ii_{k})}_h$, $
f^{(\ii_{k},\ii_{k})}_h$ and  $  f^{(\ii_{k+1},\ii_{k})}_h$ are null.
\end{proof}

\section{Extension to non-uniform knots and (partially) matching meshes}\label{appendix:a}

Note that one can extend the presented construction easily to multi-patch domains 
with non-uniform meshes and partially matching interfaces. We will briefly sketch the necessary adaptions. 
We assume to have different spline spaces $\mathcal{S}^{(i)}$ for every patch $\Omega^{(i)}$ with $i\in\indexOmega$. 
Every space satisfies 
$$
  \mathcal{S}^{(i)} = \mathcal{S}^{(i)}_1 \otimes \mathcal{S}^{(i)}_2,
$$
where $\mathcal{S}^{(i)}_k$ is a univariate spline space of degree $p$ and regularity $r$, having $n_k^{(i)}$ distinct inner knots 
$$
  0<\eta^{(i)}_{k,1}<\eta^{(i)}_{k,2}<\ldots<\eta^{(i)}_{k,n_k^{(i)}-1}<1,
$$
each with multiplicity $p-r$ and having $0$ and $1$ as boundary knots with multiplicity $p+1$.

Having defined different spaces for every patch, we change Assumption \ref{ass:Omega} and Definition \ref{defi:V0-V1} and assume 
$\f F^{(i)}\in \mathcal{S}^{(i)} \times \mathcal{S}^{(i)}$
as well as $f^{(i)}_h = \varphi_h\circ \f F^{(i)}\in \mathcal{S}^{(i)}$. Now, in order to have a sufficiently large 
$C^0$  isogeometric space along every interface, we need that the knot meshes are (partially) matching along all interfaces. 
\begin{assumption}
Consider an interface $\Sigma^{(i)}$, with $i\in\indexSigmaint$.
  Assume $\f F^{(\ii_1)}$, $\f F^{(\ii_2)}$  are in standard
 form for $\Sigma^{(i)}$. Then the corresponding meshes are 
\begin{itemize}
 \item matching, i.e., $\mathcal{S}^{(\ii_1)}_2 = \mathcal{S}^{(\ii_2)}_1$, or 
 \item partially matching, i.e., $\mathcal{S}^{(\ii_1)}_2 \subseteq \mathcal{S}^{(\ii_2)}_1$ or $\mathcal{S}^{(\ii_2)}_1 \subseteq \mathcal{S}^{(\ii_1)}_2$.
\end{itemize}
\end{assumption}
Note that two meshes are matching along an interface, if the corresponding knots are the same. The meshes are partially matching along an interface, 
if the knots of one patch are a subset of the knots of the other.

The space $\W$ can be constructed just as for uniform meshes. The patch-interior basis (Definition \ref{lem:patch-space}) needs 
no additional modification. The edge-interior basis (Definition \ref{def:edge-space}) uses spaces $\mathcal{S}^+$ and $\mathcal{S}^-$,
which are built from $\mathcal{S}^{(\ii_1)}_2 \cap \mathcal{S}^{(\ii_2)}_1$ by reducing the knot multiplicity by one or reducing the polynomial degree by one, respectively. 
See \cite{KaSaTa17} for a construction of the complete basis for non-uniform knots. 
The vertex basis (Definition \ref{def:vertex-basis}) is defined as a linear combination of patch and edge contributions and can be constructed 
analogously.

\section{Another approximating subspace $\widetilde{\W} \subseteq \mathcal{V}^1$} \label{sec:another_space}

Instead of the space $\W$ given in \eqref{eq:W}, one can consider a slightly larger subspace $\widetilde{\W} \subseteq \mathcal{V}^{1}$, which contains all isogeometric functions 
which are $C^2$  at the interior vertices and boundary vertices of valency~$\nu \geq 3$ and are $C^1$ everywhere else. The space~$\widetilde{\W}$ is given by
\begin{equation*}
\widetilde{\W} = 	\left(\bigoplus_{i\in\indexOmega}\widetilde{\W}^{\circ}_{\Omega^{(i)}}\right)
	\oplus\left(\bigoplus_{i\in\indexSigmaint}\widetilde{\W}^{\circ}_{\Sigma^{(i)}}\right)
	\oplus\left(\bigoplus_{i\in\indexXint \cup \widetilde{\mathcal{I}}^{\Gamma}_{\mathcal{X}}}\W_{{\f x}^{(i)}}\right),
\end{equation*}
where the indices in $\widetilde{\mathcal{I}}^{\Gamma}_{\mathcal{X}} \subseteq \indexXbound$ represents all boundary vertices of valency~$\nu \geq 3$. In contrast to~\eqref{eq:W}, 
different patch interior spaces, denoted by $\widetilde{\W}^{\circ}_{\Omega^{(i)}}$, and different edge function spaces, denoted by $\widetilde{\W}^{\circ}_{\Sigma^{(i)}}$, are used 
to generate the space~$\widetilde{\W}$. A further difference is that the new edge function space $\widetilde{\W}^{\circ}_{\Sigma^{(i)}} $ will be now only taken for all interfaces, 
and that the vertex function space~$\W_{{\f x}^{(i)}}$ have to be only
selected for all interior vertices and for all boundary vertices of
valency~$\nu \geq 3$. Below, we will present the 
definitions of the spaces $\widetilde{\W}^{\circ}_{\Omega^{(i)}}$ and $\widetilde{\W}^{\circ}_{\Sigma^{(i)}}$, which will be similar to the ones for the spaces 
$\W^{\circ}_{\Omega^{(i)}}$ and $\W^{\circ}_{\Sigma^{(i)}}$, see Definition~\ref{lem:patch-space} and \ref{def:edge-space}, respectively. 

Before, we will need some additional 
assumptions and definitions. We assume that in case of $\beta^{(i)} \equiv 0$ for $\Sigma^{(i)}$, $i \in \indexSigmaint$, the functions $\beta^{(i,\ii_1)}$ and $\beta^{(i,\ii_2)}$ are 
selected as $\beta^{(i,\ii_1)} \equiv \beta^{(i,\ii_1)} \equiv 0$. For each $\Sigma^{(i)}$, $i \in \indexSigmaint$, let 
\[
z_{\beta}^{(i)} = \{\xi_0 \in {h,\ldots, (n-1)h} \; | \; \beta^{(i)}(\xi_0) =0 \} ,
\text{ } h^{(i)}_\beta = \left\{ \begin{array}{ll}
                               0  & \text{ if } \beta^{(i)} \equiv 0, \\
                               1  & \text{ otherwise,}
                                \end{array} \right. 
\]
and
\[
 d^{(i)}_{\alpha} = \max (\deg (\alpha^{(i,\ii_1)}), \deg (\alpha^{(i,\ii_2)})).
\]
Since $\alpha^{(i,\ii_1)}$ and $\alpha^{(i,\ii_2)}$ are linear polynomials, and $\beta^{(i)}$ is a quadratic polynomial, we obtain that 
$d^{(i)}_{\alpha} \in \{0,1 \}$ and $z_{\beta}^{(i)} \in \{0,1,2, n\}$, cf. \cite{KaSaTa17}. For each $\ell \in \{ 1, \ldots, n-1 \}$, we denote by $\mathcal{S}^{p,r}_{h,\ell}$ the 
univariate spline space of degree~$p$ on the parameter domain $[0,1]$, constructed from the open knot vector with $n$ non-empty knots spans with (mesh) size $h=1/n$, where the inner 
knots $i h$, $i \in \{1, \ldots, n-1 \}$ with $i \neq \ell$, have multiplicity $p-r$, and the inner knot $\ell h$ has multiplicity $p-r+1$. This means that functions of the 
space~$\mathcal{S}_{h,\ell}^{p,r}$ are $C^r$  on $[0,1]$ except at the inner knot $\ell h$, where they are only $C^{r-1}$ .

We first define the patch interior space~$\widetilde{\W}^{\circ}_{\Omega^{(i)}} \supseteq \W^{\circ}_{\Omega^{(i)}} $. In contrast to the space 
$\W^{\circ}_{\Omega^{(i)}}$, the isogeometric functions of the space $\widetilde{\W}^{\circ}_{\Omega^{(i)}}$ need not have vanishing values and gradients at possible 
boundary edges of the multi-path domain~$\Omega$, but still have vanishing 
values and gradients at the patch interfaces. 
\begin{definition}
Let $i \in \indexOmega$, we define the space~$\widetilde{\W}^{\circ}_{\Omega^{(i)}}$ as 
\begin{equation*}
  \widetilde{\W}^{\circ}_{\Omega^{(i)}} = \text{span}\,\left\{\basisfct{i}{\f j}:
     \overline \Omega^{(i)}  \rightarrow\RR\text{ such that }
    \basisfct{i}{\f j}\circ \f F^{(i)} = b_{\f j}, \mbox{ for } \f j \in
\widetilde{\mathbb{I}}^{\circ}_{\Omega^{(i)}} \right\} 
\end{equation*}
where the index set $\widetilde{\mathbb{I}}^{\circ}_{\Omega^{(i)}}$ takes all $\f j \in \mathbb{I}$ which do not belong to the $C^1$ data of an interface $\Sigma^{(\ii)}$, with $\ii \in \indexSigmaint$, or to the $C^2$ data of a vertex $\f x^{(\ii)}$, with $\ii \in \indexXint \cup \widetilde{\mathcal{I}}^{\Gamma}_{\mathcal{X}}$.
\end{definition}

The definition of the edge function space~$\widetilde{\W}^{\circ}_{\Sigma^{(i)}} \supseteq \W^{\circ}_{\Sigma^{(i)}}$ is based on the construction of the 
(entire) $C^1$  isogeometric space for AS-$G^1$ two-patch geometries presented in \cite{KaSaTa17}.

\begin{definition}
 Let $\Sigma^{(i)}$, for $i\in\indexSigmaint$, be an interface in standard form. 
Consider  the univariate spline spaces $\widetilde{\mathcal{S}}^+ = \mathcal{S}^{p,r+h^{(i)}_{\beta}}_h$
and $\widetilde{\mathcal{S}}^- = \mathcal{S}^{p-d^{(i)}_{\alpha},r}_h$, with bases $\{ \widetilde{b}^+_j\}_{j \in \widetilde{\mathbb{I}}^+}$, and $\{\widetilde{b}^-_j\}_{j
  \in\widetilde{\mathbb{I}}^-}$, respectively, where 
$\widetilde{\mathbb{I}}^\pm=\{0,\ldots,\widetilde{N}^\pm-1\}$ with $\widetilde{N}^+= (p-r-h_{\beta}^{(i)})(n-1)+p+1$ and 
$\widetilde{N}^-=(p-d_{\alpha}^{(i)}-r)(n-1)+p-d_{\alpha}^{(i)}+1$. For each $\ell \in \{1, \ldots, n-1 \}$, let $\widetilde{b}^{\#}_{\ell}$ be a B-spline of the 
space $\mathcal{S}_{h,\ell}^{p,r}$ with the property $\widetilde{b}^{\#}_{\ell}(\ell h) \neq 0$ which have vanishing derivatives up to second order at both interface vertices. We define the index set
\[
 \widetilde{\mathbb{I}}^{\#}= \left\{ \begin{array}{ll}
                               \emptyset  & \text{ if } z_{\beta}^{(i)}=0 \text{ or } \beta \equiv 0, \\
                               \{\ell \in \{1, \ldots, n-1 \} \; |  \; \beta^{(i)}(\ell h) =0  \mbox{ and }\widetilde{b}^{\#}_{\ell}\mbox{ exists}\}& \text{ otherwise,}
                                \end{array} \right.
\]
and $\widetilde{\mathbb{I}}_{\Sigma^{(i)}}= (\widetilde{\mathbb{I}}^{+} \times \{0\}) \cup 
(\widetilde{\mathbb{I}}^{-} \times \{1\}) \cup (\widetilde{\mathbb{I}}^{\#} \times \{ 2\})$.
In addition, let $c_{0}=b_{0}+b_{1}$ and $c_{1}=\frac{h}{p}b_1$. We define the space $\widetilde{\W}_{\Sigma^{(i)}}$ as 
\begin{equation*}\label{eq:WSigma0Tilde}
  \widetilde{\W}_{\Sigma^{(i)}} = 
  \mbox{span}\, \left\{\basisfctSigmaTilde{i}{\f j}: \overline \Omega^{(\ii_1)}\cup \overline
\Omega^{(\ii_2)}  \rightarrow \mathbb{R}, \text { for }\f j \in \widetilde{\mathbb{I}}_{\Sigma^{(i)}}\right\},
\end{equation*} 
where
\begin{equation*}\label{eq:bj-edge-traceTilde}
  \begin{aligned}
    \basisfctSigmaTilde{i}{(j_1,0)}\circ \f F^{(\ii_1)}(\xi_1,\xi_2)&=
 \widetilde{b}^+_{j_1}(\xi_2)\ibasis_0(\xi_1) -  \beta^{(i,\ii_1)}(\xi_2) (\widetilde{b}^+_{j_1})'(\xi_2) \ibasis_1(\xi_1) ,\\
   \basisfctSigmaTilde{i}{(j_1,0)}\circ \f F^{(\ii_2)}(\xi_1,\xi_2)&=
 \widetilde{b}^+_{j_1}(\xi_1)\ibasis_0(\xi_2) -  \beta^{(i,\ii_2)}(\xi_1) (\widetilde{b}^+_{j_1})'(\xi_1) \ibasis_1(\xi_2) ,
  \end{aligned}
\end{equation*}
for $j_1 \in \widetilde{\mathbb{I}}^{+}$,
\begin{equation*}\label{eq:bj-edge-derivativetilde}
  \begin{aligned}
    \basisfctSigmaTilde{i}{(j_1,1)}\circ \f F^{(\ii_1)}(\xi_1,\xi_2)&=\alpha^{(i,\ii_1)}(\xi_2) \widetilde{b}^-_{j_1}(\xi_2) {b_1}(\xi_1) ,\\
   \basisfctSigmaTilde{i}{(j_1,1)}\circ \f F^{(\ii_2)}(\xi_1,\xi_2)&= - \alpha^{(i,\ii_2)}(\xi_1) \widetilde{b}^-_{j_1}(\xi_1) {b_1}(\xi_2) ,
  \end{aligned}
\end{equation*}
for $j_1 \in \widetilde{\mathbb{I}}^{-}$, and
\begin{equation*}\label{eq:bj-edge-Additional}
  \begin{aligned}
    \basisfctSigmaTilde{i}{(j_1,2)}\circ \f F^{(\ii_1)}(\xi_1,\xi_2)&=
 \widetilde{b}^{\#}_{j_1}(\xi_2)\ibasis_0(\xi_1) -  \beta^{(i,\ii_1)}(\xi_2) (\widetilde{b}^{\#}_{j_1})'(\xi_2) \ibasis_1(\xi_1)  + 
 \frac{\beta^{(i,\ii_1)}(j_1 h)}{\alpha^{(i,\ii_1)}(j_1 h)}  \alpha^{(i,\ii_1)}(\xi_2) (\widetilde{b}^{\#}_{j_1})'(\xi_2) \ibasis_1(\xi_1),\\
   \basisfctSigmaTilde{i}{(j_1,2)}\circ \f F^{(\ii_2)}(\xi_1,\xi_2)&=
 \widetilde{b}^{\#}_{j_1}(\xi_1)\ibasis_0(\xi_2) -  \beta^{(i,\ii_2)}(\xi_1) (\widetilde{b}^{\#}_{j_1})'(\xi_1) \ibasis_1(\xi_2) 
  +  \frac{\beta^{(i,\ii_2)}(j_1 h)}{\alpha^{(i,\ii_2)}(j_1 h)} \alpha^{(i,\ii_2)}(\xi_1) (\widetilde{b}^{\#}_{j_1})'(\xi_1) \ibasis_1(\xi_2),
  \end{aligned}
\end{equation*}
for $j_1 \in \widetilde{\mathbb{I}}^{\#}$. Let then $\widetilde{\W}^{\circ}_{\Sigma^{(i)}}$ be the subspace of $\widetilde{\W}_{\Sigma^{(i)}}$, given by 
\begin{equation*}
  \widetilde{\W}^{\circ}_{\Sigma^{(i)}} = 
  \{ \varphi_h\in \widetilde{\W}_{\Sigma^{(i)}}: \partial_{x_1}^{m_1}\partial_{x_2}^{m_2}\varphi_h(\f x^{(\ii)})=0 \mbox{ for all }\ii\in \indexXint \cup \widetilde{\mathcal{I}}^{\Gamma}_{\mathcal{X}} \mbox{ and } m_1,m_2 \geq 0 \mbox{ and }m_1+m_2 \leq 2\}.
\end{equation*} 
\end{definition}

\begin{remark}
In contrast to the subspace~$\W$, the dimension of the subspace~$\widetilde{\W}$ depends on the domain parametrization. Let $\widehat{\W}^{\circ}_{\Omega^{(i)}} = \widetilde{\W}^{\circ}_{\Omega^{(i)}}$ and let
\begin{equation*}
  \widehat{\W}^{\circ}_{\Sigma^{(i)}} = 
  \{ \varphi_h\in {\W}_{\Sigma^{(i)}}: \partial_{x_1}^{m_1}\partial_{x_2}^{m_2}\varphi_h(\f x^{(\ii)})=0 \mbox{ for all }\ii\in \indexXint \cup \widetilde{\mathcal{I}}^{\Gamma}_{\mathcal{X}}, \mbox{and } m_1,m_2 \geq 0 \mbox{ and }m_1+m_2 \leq 2\}.
\end{equation*}  
Then, the subspace
\[
\widehat{\W} = 	\left(\bigoplus_{i\in\indexOmega}\widehat{\W}^{\circ}_{\Omega^{(i)}}\right)
	\oplus\left(\bigoplus_{i\in\indexSigmaint}\widehat{\W}^{\circ}_{\Sigma^{(i)}}\right)
	\oplus\left(\bigoplus_{i\in\indexXint \cup \widetilde{\mathcal{I}}^{\Gamma}_{\mathcal{X}}}\W_{{\f x}^{(i)}}\right)
\]
with $\W \subseteq \widehat{\W} \subseteq \widetilde{A} \subseteq \V^{1} $, would be another choice of a $C^1$  isogeometric subspace, and its dimension is as the dimension of 
the subspace~$\W$ independent of the domain parametrization.
\end{remark}

\section*{References}

\end{document}